\documentclass[12]{amsart}
\usepackage{amssymb}
\usepackage{amsfonts}
\usepackage{latexsym}   
\usepackage{amssymb}    
\usepackage[leqno]{amsmath}    
\usepackage{amsbsy}
\usepackage{hyperref}
\usepackage{amsthm}
\usepackage{amsgen}
\usepackage{amsfonts}
\usepackage{array}
\usepackage{epsfig}
\usepackage{cleveref}
\usepackage{psfrag}
\usepackage{scalerel}
\usepackage{tikz}
\usetikzlibrary{shapes.geometric}

\theoremstyle{plain}

\newtheorem{algorithm}{Algorithm}[section]

\newtheorem{conjecture}[algorithm]{Conjecture}
\newtheorem{corollary}[algorithm]{Corollary}

\newtheorem{definition}[algorithm]{Definition}
\newtheorem{example}[algorithm]{Example}

\newtheorem{lemma}[algorithm]{Lemma}
\newtheorem{question}[algorithm]{Question}

\newtheorem{theorem} [algorithm] {Theorem}
\newtheorem{main} [algorithm] {Main Theorem}

\newtheorem{proposition}[algorithm]{Proposition}
\newtheorem{remark}[algorithm]{Remark}

\numberwithin{equation}{algorithm}

\newtheorem{CSL}[algorithm]{Chang Skjelbred Lemma}

\def\qqq{\mathbb{Q}}
\def\rrr{\mathbb{R}}
\def\ccc{\mathbb{C}}
\def\zzz{\mathbb{Z}}

\DeclareMathOperator{\Fix}{Fix}
\def\codim{\textrm{codim}}
\def\bdm{\begin{displaymath}}
\def\edm{\end{displaymath}}
\def\beq{\begin{equation}}
\def\eeq{\end{equation}}
\def\bes{\begin{equation*}}
\def\ees{\end{equation*}}
\def\epcm{\end{picture}\end{center}\end{minipage}}
\def\bpcm{\begin{minipage}{80pt}\begin{center}\begin{picture}}

\def\t2{T^2}

\def\f4{F_4}
\def\g2{G_2}

\def\mft{\mathfrak{t}}

\def\p2{\frac{\pi}{2}}

\def\Fix{\mathrm{Fix}}
\def\rk{\mathrm{rk}}

\def\dim{\mathrm{dim}}

 \numberwithin{equation}{section}
  \numberwithin{figure}{section}

\newtheorem*{ack}{Acknowledgements}

\begin{document}
\newcommand{\comment}[1]{\vspace{5 mm}\par \noindent
\marginpar{\textsc{Note}}
\framebox{\begin{minipage}[c]{0.95 \textwidth}
#1 \end{minipage}}\vspace{5 mm}\par}

\title[Odd-dimensional GKM-manifolds of non-negative curvature]{Odd-dimensional GKM-manifolds of non-negative curvature}

\author[Escher]{Christine Escher}
\address[Escher]{Department of Mathematics, Oregon State University, Corvallis, Oregon}
\email{tine@math.orst.edu}

\author[Goertsches]{Oliver Goertsches}
\address[Goertsches]{Fachbereich Mathematik und Informatik der Philipps-Universit\"at Marburg, Germany}
\email{goertsch@mathematik.uni-marburg.de}

\author[Searle]{Catherine Searle}
\address[Searle]{Department of Mathematics, Statistics, and Physics, Wichita State University, Wichita, Kansas}
\email{searle@math.wichita.edu}

\subjclass[2000]{Primary: 53C20; Secondary: 57S25} 

\date{\today}

\begin{abstract}

\bigskip

Let $M$ be a closed, odd GKM$_3$ manifold of non-negative sectional curvature. We show that in this situation one can associate an ordinary abstract GKM$_3$ graph to $M$ 
and prove that if this graph is orientable, then both the equivariant and the ordinary rational cohomology of $M$ split off the cohomology of an odd-dimensional sphere.
\end{abstract}

\maketitle

\section{Introduction}

A long-standing problem in Riemannian geometry is the classification of positively and non-negatively curved manifolds. One characteristic shared by many of the known examples is a high degree of symmetry. The
 Grove Symmetry Program suggests we attempt the classification of such manifolds with the additional hypothesis of 
 ``large" symmetries. The eventual goal of this program is to be able to eliminate the hypothesis of symmetries entirely.
 
 A natural first step is to consider the case of abelian symmetries. 
 For the case of positive curvature, results due to Grove and Searle \cite{GS}, Rong \cite{R} and Fang and  Rong \cite{FR}, and  Wilking \cite{Wi}  give us a classification up to  diffeomorphism, homeomorphism, or rational homotopy equivalence for a $T^k$-action, provided $k$ equals $\lfloor (n+1)/2 \rfloor$,  $\lfloor (n-1)/2 \rfloor$, or is greater than or equal to $\lfloor (n+4)/4 \rfloor$, respectively. For non-negative curvature,  an equivariant diffeomorphism classification for dimensions less than or equal to nine for $T^k$-actions with $k=\lfloor 2n/3\rfloor$ follows from work of  Galaz-Garc\'ia and  Searle \cite{GGS1},  Galaz-Garc\'ia and  Kerin \cite{GGK}, and  Escher and  Searle \cite{ES2}. A diffeomorphism classification for dimensions less than or equal to $6$ for $T^k$-actions with $k=\lfloor 2n/3\rfloor-1$ follows from work of Kleiner \cite{K} and Searle and Yang \cite{SY}, Galaz-Garc\'ia and Searle \cite{GGS2}, and Escher and  Searle \cite{ES1}. Note that all of these results rely heavily on the 
 existence of fixed point sets of ``small" codimension.
 
 From this point of view, one may consider GKM$_k$ manifolds as occupying the other end of the spectrum. Note that these are $2n$-dimensional manifolds with a torus action of rank $\leq n$.   A consequence of the GKM$_k$ condition is that for a $T^m$-action on $M^{2n}$, for all $l\leq \min(k, m)-1$, 
 there exist codimension $l$ torus subgroups of $T^m$ fixing $2l$-dimensional submanifolds of $M$, which we denote by $N^{2l}$. Furthermore, with the induced $T^l$ action, each $N^{2l}$ is a torus manifold, that is an orientable, even-dimensional manifold such that $T^l$ has non-empty fixed point set. GKM$_k$  manifolds of both positive and non-negative curvature were studied by Goertsches and Wiemeler in \cite{GW1, GW2}, respectively, where they showed that the GKM$_3$, respectively GKM$_4$, condition allows them to classify such manifolds up to real, respectively, rational cohomology type. The notion of GKM manifold was extended to odd dimensions by He \cite{H}. We  call such manifolds {\em odd GKM manifolds}. Inspired by the work of \cite{GW1, GW2} and \cite{H}, we consider the case of odd GKM$_3$ manifolds of  positive and non-negative curvature. 
 
 Our main result concerns odd GKM$_3$ manifolds of non-negative sectional curvature.   Throughout this article, we only consider cohomology with rational coefficients, and note that an  abstract GKM$_3$ graph  is {\em orientable} provided  its top cohomology class is non-trivial.

   \begin{main}\label{main} Let $M^{2n+1}$ be a closed, non-negatively curved odd GKM$_3$ manifold, $\bar{\Gamma}_M$ the GKM$_3$ graph of $M$, and $k$ the number of floating edges at a vertex of $\bar{\Gamma}_M$.  Suppose that $(\Gamma,\alpha,\nabla)$,
   the abstract, even-dimensional GKM$_3$ graph obtained from $\bar{\Gamma}_M$, is orientable. Then $H^*(M^{2n+1})$  splits off the cohomology ring of an odd dimensional sphere, that is,
  \[
H^*(M)\cong H^*(\Gamma,\alpha, \nabla)\otimes H^*(S^{2k+1}).
\]
 \end{main}

In the process of proving Theorem \ref{main}, we also obtain a similar result for the equivariant cohomology of $M$, see Theorem \ref{main1}.   For the definition of floating edges of a GKM graph of an odd-dimensional GKM manifold, see Definition \ref{edge}.  
 
 In Theorem \ref{main}, the abstract, even-dimensional GKM$_3$ graph, $\Gamma$, that we obtain from the odd GKM$_3$ graph associated to $M^{2n+1}$  has two-dimensional faces that contain at most $4$ vertices.
 The rational cohomology of a GKM$_k$ manifold is completely determined by its corresponding vertex-edge graph, $\Gamma$.   Unfortunately,  there is no general classification of even-dimensional GKM$_3$ graphs whose two dimensional faces contain at most $4$ vertices. That is,  it is as yet unknown whether every such GKM$_3$ graph corresponds to a closed, non-negatively curved GKM$_3$ manifold. 
  However, if there are no quadrangles as two-dimensional faces of $\Gamma$, that is, if no two-dimensional face in the odd-dimensional GKM graph of $M$ is of the form $(4)$ in Theorem \ref{graphs}, then by the main result of \cite{GW1}, $H^*(\Gamma,\alpha,\nabla)$ is isomorphic to the real cohomology ring of a compact rank one symmetric space (CROSS). Since the result in \cite{GW1} 
 was obtained via a classification of all possible GKM$_3$ graphs and by applying the GKM theorem, the result also holds for rational coefficients. We obtain the following theorem.

 \begin{theorem}\label{2}  Let $M^{2n+1}$ be a closed, non-negatively curved,  odd GKM$_3$ manifold.  
 Suppose that the two-dimensional faces in $\bar{\Gamma}_M$, the odd-dimensional GKM graph of $M$, are not of the form $(4)$ in Theorem \ref{graphs} and that $(\Gamma,\alpha,\nabla)$,
   the abstract, even-dimensional GKM$_3$ graph obtained from $\bar{\Gamma}_M$, is orientable. Then
  $H^*(M^{2n+1})$ is the tensor product of the rational cohomology ring of an odd dimensional sphere,  and a simply-connected CROSS, that is,
 \[
H^*(M)\cong H^*(N^{2n-2k})\otimes H^*(S^{2k+1}),
\]
where $N^{2n-2k}$ is a simply-connected CROSS.
 \end{theorem}

 Using Theorem 1.4 of \cite{GW2}, we see that if we assume that our manifold is GKM$_4$, then the cohomology ring of the manifold splits as that of an odd-dimensional sphere and a finite quotient of a non-negatively curved torus manifold.
 \begin{theorem}\label{3} Let $M^{2n+1}$ be a closed,  non-negatively curved odd GKM$_4$ manifold. Suppose that  $(\Gamma,\alpha,\nabla)$,
   the abstract, even-dimensional GKM$_3$ graph obtained from $\bar{\Gamma}_M$, the odd-dimensional GKM graph of $M$, is orientable. Then 
  $H^*(M^{2n+1})$ is the tensor product of the cohomology ring of an odd dimensional sphere,  and the cohomology ring of a (quotient of) a torus manifold, that is
 \[
H^*(M)\cong H^*(N^{2n-2k}/G)\otimes H^*(S^{2k+1}),
\]
where $N$ is  a simply-connected, non-negatively curved torus manifold and $G$ is a finite group acting isometrically (and orientably) on $N$.
 \end{theorem}

To date, only odd GKM graphs without signs have been treated in the literature. Adding the restriction that the manifold admit a $T$-invariant almost contact structure, that is, an almost contact structure invariant under a torus action, allows us to talk about odd GKM graphs with signs.
Using Theorem 1.5 of \cite{GW2}, we obtain the following result.

\begin{theorem}\label{AC} Let $M^{2n+1}$ be a closed,  non-negatively manifold with an odd  GKM$_4$ $T$-action which admits a $T$-invariant, alternating, almost contact structure. Suppose that  $(\Gamma,\alpha,\nabla)$,
   the abstract, even-dimensional GKM$_4$ graph obtained from $\bar{\Gamma}_M$, the odd-dimensional GKM graph of $M$, is orientable. Then the rational cohomology ring of $M$ is isomorphic to the tensor product of the rational cohomology ring of an odd-dimensional sphere and the rational cohomology ring of a generalized Bott manifold.
\end{theorem}

See Definition \ref{alt/bal} for a definition of an {\em alternating} almost contact structure. 
Finally, 
 we also obtain a full rational cohomology classification for positively curved odd GKM$_3$ manifolds, as follows.
 \begin{theorem}\label{positive} Let $M^{2n+1}$ be a closed, positively curved, odd GKM$_3$ manifold, then $M^{2n+1}$ has the rational cohomology ring of $S^{2n+1}$.
\end{theorem}

We note that of the known examples of odd-dimensional manifolds of positive curvature, the only ones that are rational homology spheres, but not diffeomorphic to spheres, are the so-called Berger spaces,  $B^7=SO(5)/SO(3)$ and $B^{13}=SU(5)/(Sp(2)\times_{\mathbb{Z}_2} S^1)$. These two manifolds admit both an invariant  transitive and an invariant cohomogeneity one action, but, the maximal torus of both actions does not satisfy the requirements to be odd GKM$_3$, since the maximal torus of the corresponding $G$-action in either case does not have fix points.
 
Finally, we point out that the positive and non-negative curvature hypotheses in our results serve to restrict the number of vertices in the $2$-dimensional faces of  the odd GKM$_k$ graph associated to the manifold. Thus, all of the above theorems can be reframed in a curvature-free setting by simply assuming the corresponding restrictions on these graphs.

\subsection{Organization}
The paper is organized as follows. We include basic notation and preliminary material in Section \ref{s2}. In Section \ref{s3}, we classify the universal covers of closed, non-negatively curved $5$-dimensional GKM manifolds. 
In Section \ref{s4}, we classify the corresponding graphs of these $5$-manifolds. In Section \ref{s5}, we prove Theorems \ref{main}, \ref{2}, \ref{3}, and \ref{positive}. In Section \ref{s6}, we give the proof of Theorem \ref{AC}.

\begin{ack} We are grateful to a referee for bringing to our attention  the example in Subsection
\ref{ss5}. The first and third authors are grateful to the Max Planck Institute for Mathematics in Bonn for its hospitality and support during the summer of 2017, where this work was initiated. C. Escher gratefully acknowledges partial support  from the Simons Foundation  (\#585481, C. Escher).
C. Searle gratefully acknowledges partial support  from grants from the National Science Foundation (\#DMS-1611780 and \#DMS-1906404), as well as  from the Simons Foundation (\#355508, C. Searle).

\end{ack}
\section{Preliminaries}\label{s2}

In this section we  gather basic results and facts about transformation groups, equivariant cohomology, even- and odd-dimensional GKM and GKM$_k$ theory, as well as results concerning $G$-invariant manifolds of non-negative sectional curvature.

\subsection{Transformation Groups}

Let $G$ be a compact Lie group acting on a smooth manifold $M$. We denote by $G_x=\{g\in G \mid gx=x\}$ the \emph{isotropy group} at $x\in M$ and by $G(x)=\{gx \mid g\in G\}$ the \emph{orbit} of $x$.  Note that $G(x)$ is homeomorphic to  $G/G_x$ since $G$ is compact.   We denote the orbit space of the $G$-action by $M/G$ and note that if $M$ admits a lower sectional curvature bound and the $G$-action is isometric, then $M/G$ is an Alexandrov space admitting the same lower curvature bound. We  denote the fixed point set of $M$ by $G$  as either 
$M^G$ or $\Fix(M; G)$, using whichever may be more convenient.

One measurement for the size of a transformation group $G\times M\rightarrow M$ is the dimension of its orbit space $M/G$, also called the {\it cohomogeneity} of the action. This dimension is clearly constrained by the dimension of the fixed point set $M^G$  of $G$ in $M$. In fact, $\dim (M/G)\geq \dim(M^G) +1$ for any non-trivial, non-transitive action. In light of this, the {\it fixed point cohomogeneity} of an action, denoted by $\textrm{cohomfix}(M;G)$, is defined by
\[
\textrm{cohomfix}(M; G) = \dim(M/G) - \dim(M^G) -1\geq 0.
\]
A manifold with fixed point cohomogeneity $0$ is also called a {\it $G$-fixed point homogeneous manifold}.

We now recall Theorem I.9.1  of Bredon \cite{Br} which characterizes how to lift a group action to a covering space.

\begin{theorem}\cite{Br}\label{lift}
 Let $G$ be a connected Lie group acting effectively on a connected, locally path-connected space $X$ and let $X'$ be any covering space of $X$. Then there is a covering group $G'$ of $G$ with an effective action of $G'$ on $X'$ covering the given action. Moreover, $G'$ and its action on $X'$ are unique.
 
The kernel of $G'\rightarrow G$ is a subgroup of the group of deck transformations of $X'\rightarrow X$. In particular, if $X'\rightarrow X$
 has finitely many sheets, then so does $G'\rightarrow G$.
If $G$ has a fixed point in $X$, then $G' = G$ and $\Fix(X'; G)$ is the
full inverse image of $\Fix(X; G)$.
\end{theorem}

The presence of a $G$-action on a manifold $M$ induces a topological stratification of $M$. In particular, when the group is a torus, that is, $G=T$, 
 a special name is given to the strata of the $T$-action.
\begin{definition}[{\bf The k-skeleton of ${\bf M}$}] For $G=T$, a torus, we define the {\em k-skeleton of M} to be 
 $$M_k = \{p\in M\mid \dim(T(p))\leq k\}.$$  We then obtain a $T$-invariant topological stratification of $M$ as follows: $M_0\subset M_1\subset \cdots
\subset M_{\dim(T)} = M$ on $M$, where the $0$-skeleton $M_0$ is exactly the fixed point set $M^T$.
\end{definition}
\subsection{Equivariant cohomology}

We begin by providing some basic information about  equivariant cohomology and equivariantly formal manifolds for torus actions.
\begin{definition}[{\bf Equivariant Cohomology}]
Given an action of a torus $T$ on a compact manifold $M$, the \emph{equivariant cohomology} of the action is defined as
\[
H^*_T(M) = H^*(M\times_T ET),
\]
where $ET\to BT$ is the classifying bundle of $T$ and $ET$ is a contractible space on which $T$ acts freely.
\end{definition}
The equivariant cohomology has the natural structure of an $H^*(BT)$-algebra, via the projection $M\times_T ET \to BT$. Note that $H^*(BT)$ is isomorphic to the ring of rational polynomials on the Lie algebra $\mft$, in the following sense. Denoting the rational points in $\mft$, that is, the tensor product of the integer lattice in $\mft$ with $\qqq$, by ${{\mathfrak{t}}_{\mathbb{Q}}}$, then $H^*(BT)$, a rational polynomial ring in $\dim(T)$ variables, is isomorphic to  $S({{\mathfrak{t}}_{\mathbb{Q}}^*})$, the symmetric algebra over  ${\mathfrak{t}}_{\mathbb{Q}}^*$.

Given an action of a torus $T$ on $M$, we may compare  $H^*_T(M)$ with
$H^*_T(M^T)$ using the Borel Localization Theorem (see, for example, Corollary 3.1.8 in Allday and Puppe \cite{AP}).

\begin{theorem} [Borel Localization Theorem] \label{BLT} The restriction map
$$H^*_T(M)\rightarrow H^*_T(M^T)$$
is an $H^*(BT)$-module isomorphism modulo $H^*(BT)$-torsion.
\end{theorem}

Using this localization theorem, it is clear that if  $H^*_T(M)$ is actually  a free $H^*(BT)$-module, one can hope for a stronger relation between  the manifold $M$ and its
fixed point set $M^T$. This motivates the following definition. 

\begin{definition}[{\bf Equivariantly Formal}]
We say that an action of a torus $T$ on $M$ is \emph{equivariantly formal} if $H^*_T(M)$ is a free $H^*(BT)\cong S({\mathfrak{t}}_{\mathbb{Q}}^*)$-module. \end{definition}

As a compact manifold has finite-dimensional cohomology, it follows from Corollary 4.2.3 of \cite{AP} that equivariant formality is equivalent to the degeneration of the Leray-Serre spectral sequence of the Borel fibration $M\hookrightarrow M\times_{T} ET\rightarrow BT$ at the $E_2$-term. 
Moreover, the following are some well-known and important properties of equivariantly formal actions.
\begin{proposition}\label{finteMT}
An action of a torus $T$ on a compact manifold M with $H^{\mathrm{odd}}(M) = \{0\}$ is automatically equivariantly formal. The converse implication is true provided that the $T$-fixed point set $M^T$ is finite.
\end{proposition}
The first statement of Proposition \ref{finteMT} follows because if $H^{\mathrm{odd}}(M)=\{0\}$ then the spectral sequence degenerates at the $E_2$-term. The second statement is a consequence of the Borel Localization Theorem, as then $H^*_T(M)\cong H^*(BT)\otimes H^*(M)$ injects into the module $H^*_T(M^T) \cong H^*(BT)\otimes H^*(M^T)$ which vanishes in odd degrees. The next proposition is Theorem 3.10.4 in \cite{AP}.

\begin{proposition}\label{eformal}\cite{AP}
For any action of a torus $T$ on a compact manifold $M$, we have $\dim\, H^*(M^T)\leq \dim\, H^*(M)$. Equality holds if and only if the action is equivariantly formal.
\end{proposition}
The Leray-Hirsch theorem implies that for equivariantly formal actions the ordinary cohomology ring is encoded in the equivariant cohomology algebra.
\begin{proposition}\label{Ecohom}
For an equivariantly formal action of a torus $T$ on a compact manifold $M$ the natural map $H^*_T(M)\to H^*(M)$ is surjective and induces an isomorphism of $\rrr$-algebras
\[
\frac{H^*_T(M)}{S^+({\mathfrak{t}}_{\mathbb{Q}}^*)\cdot H^*_T(M)}\cong H^*(M),
\]
where $S^+({\mathfrak{t}}_{\mathbb{Q}}^*)$ denotes the ideal in $S({\mathfrak{t}}_{\mathbb{Q}}^*)$ generated by real-valued polynomials of positive degree.
\end{proposition}

For equivariantly formal actions, the Borel Localization Theorem \ref{BLT} gives an embedding of $H^*_T(M)$ into $H^*_T(M^T)$. The image of this embedding can be described as follows, combining Lemma 2.3 of Chang and Skjelbred \cite{CS} for the first isomorphism with  the version of the same lemma given in Theorem 11.51 of Guillemin and Sternberg \cite{GuSt} for the second isomorphism.

\begin{CSL}
\cite{CS}, \cite{GuSt}
\label{CSLemma}  If a $T$-action on $M$ is equivariantly formal, then the
equivariant cohomology $H^*_T(M)$ only depends on the fixed point set $M^T$ and the $1$-skeleton $M_1$:
$$H^*_T(M)\cong {\mathrm{im}}\, (H^*_T(M_1)\to H^*_T(M^T))\cong  \bigcap {\mathrm{im}}(H^*_T(M^K)\rightarrow H^*_T(M^T)),$$
where the intersection is taken over all corank-1 subtori $K$ of $T$.
\end{CSL}
 Moreover, equivariant formality is {\em inherited} by subtori of the $T$-action. More precisely, we have the following well-known proposition and corollary, see, for example \cite{H}. 
\begin{proposition}\label{ef} If a $T$-action on $M$ is equivariantly formal, then for any subtorus $K$ of $T$, both the $K$-subaction on $M$ and the induced $T/K$-action on  $M^K$ are equivariantly
formal.
\end{proposition}
\begin{corollary}\label{He} If a $T$-action on $M$ is equivariantly formal, then for any subtorus
$K$ of $T$, every connected component of $M^K$ has $T$-fixed points.
\end{corollary}

\subsection{Even-dimensional GKM theory}\label{sec:evendimGKM}
The class of manifolds, now referred to as {\em GKM manifolds}, were first introduced in the seminal work of Goresky, Kottwitz, and MacPherson \cite{GKM} to study the relation between  equivariant cohomology and ordinary cohomology, and are named for these authors. The GKM Theorem \ref{thm:GKM} (cf.\ Theorem 1.2.2 of \cite{GKM}), states that over an appropriate coefficient ring $R$,  the equivariant cohomology ring of a GKM manifold $M$ can be computed via its $1$-skeleton  and the isotropy information of the GKM torus action.
Motivated by their work, the concepts of GKM manifold and  GKM graph were introduced by Guillemin and Zara in \cite{GZ} to build a bridge between the topology and the combinatorics of these spaces.  We begin with the formal definition of a GKM manifold.

\begin{definition}[{\bf GKM Torus Action and Manifold}]\label{evenGKM}
We say that the effective action of a torus $T^l=T$, $l\leq n$, on an orientable, compact, connected manifold $M^{2n}$ is {\em GKM}, and $M^{2n}$ is called a {\em GKM manifold}, if 
\begin{enumerate}
\item the fixed point set $M^T$ of the action is finite;
\item for every $p\in M^T$, the weights $\alpha_{i, p}\in {{\mathfrak{t}}_{\mathbb{Q}}}^*/\{\pm 1\}, i=1, \hdots, n,$ of the isotropy representation of $T$ on $T_pM$ are pairwise linearly independent; and
\item the $T$-action is equivariantly formal.
\end{enumerate}

\end{definition}

We note that the original definition in \cite{GZ} does not require  the manifold to be either orientable or the $T$-action to be equivariantly formal, rather, both are assumed (sometimes implicitly) as separate hypotheses in their theorems. We include both hypotheses in our definition, since both are included in the definition of a GKM manifold in most of the recent literature (see, for example, Goertsches and Wiemeler \cite{GW1, GW2} and Kuroki \cite{Ku}).

By Proposition \ref{finteMT}, Conditions $1$ and $3$ imply the vanishing of the odd-dimensional rational cohomology groups of $M$. 
Condition $2$ is equivalent to the condition that $M_1$, the $1$-skeleton of $M$,  consists of a disjoint union of $T$-invariant, orientable submanifolds, each of which
is either fixed point free, or an embedded $S^2$. By Corollary \ref{He}, 
Condition $3$ implies that $M_1$ consists entirely of $T$-invariant embedded 2-spheres. Moreover, Conditions $1$ and $3$ combined with Proposition \ref{eformal} tells us that each such $S^2$ contains
exactly two $T$-fixed points.

 \begin{definition}[{\bf GKM$_k$ Torus Action and Manifold}]
We say that the effective action of a torus $T$ on an orientable, compact, connected manifold $M^{2n}$ is {\em GKM$_k$}, and we call $M^{2n}$ a {\em GKM$_k$ manifold}, if 
\begin{enumerate}
\item $M^{2n}$ is GKM,  and 
\item for each $p\in M^T$  any set of $k$ weights, $\alpha_{i, p}\in {{\mathfrak{t}}_{\mathbb{Q}}}^*/\{\pm 1\}, i=1, \hdots, n,$ of the isotropy representation of $T$ on $T_pM$  is linearly independent.
\end{enumerate}

\end{definition} 

\begin{remark} A GKM$_{k}$ manifold is GKM$_l$ for all $2\leq l\leq k$, and a GKM$_2$ manifold is a GKM manifold. 
\end{remark}

By convention, the $T^1$-action on $S^2$ is considered to be a GKM$_2$ manifold.
Also remark that the linear independence is well-defined for elements that are only defined up to sign. 
Condition $2$ is equivalent to  $M_{k-1}$, the ($k-1$)-skeleton, 
 being a union of $T$-invariant submanifolds, which are either are fixed point free or are
($2k-2$)-dimensional and have $T$-fixed points. Observe that by Corollary \ref{He}, the definition implies that $M_{k-1}$ consists entirely of ($2k-2$)-dimensional $T$-invariant submanifolds.

Once we have established some conventions,  we  define  the notion of an {\em abstract} GKM$_k$ graph, which consists of a triple: a graph, an axial function, and a connection.
 However, in Section \ref{s4} we will be working with  {\em geometric} GKM$_k$ graphs, that is,  graphs that  correspond to the graph of a  GKM$_k$ manifold.  We note that an abstract GKM$_k$ graph may not correspond to the graph of a  GKM$_k$ manifold and that the graph described in our Main Theorem \ref{main} is abstract.
  
We  employ the following conventions when speaking about abstract graphs. Given an abstract graph $\Gamma$, we denote by 
\begin{itemize}
\item $E(\Gamma)$  its set of oriented edges; and by 
\item $V(\Gamma)$ its set of vertices. 
\end{itemize} We  always assume that both the edge and vertex sets are finite, and we  allow multiple edges between vertices. 
For an edge $e\in E(\Gamma)$ we denote by
\begin{itemize}
\item $\bar e$ the edge with opposite orientation; 
\item $i(e)$ its initial vertex; and 
\item $t(e)$ its terminal vertex. 
\end{itemize} 
We assume that for any edge $e$, $i(e)\neq t(e)$, that is, an edge cannot connect a vertex to itself. For a vertex $v\in V(\Gamma)$ the set of edges emanating from $v$ is denoted by $E_v$. 

Before we can define an abstract  GKM$_k$ graph, we define 
 a connection, $\nabla$, on a graph $\Gamma$, as in \cite{GZ}. 
\begin{definition}[{\bf Connection}]\label{defnconnection}
A \emph{connection} on a graph $\Gamma$ is a collection $\nabla$ of maps $\nabla_e:E_{i(e)}\to E_{t(e)}$, for each $e\in E(\Gamma)$, such that
\begin{enumerate}
\item $\nabla_e(e) = \bar e$ and
\item $\nabla_{\bar e} = (\nabla_e)^{-1}$,
\end{enumerate}
for all $e\in E(\Gamma)$.
\end{definition}

 We are now in a position to define an {\em abstract GKM$_k$ graph}. We note that abstract GKM graphs were originally defined in \cite{GZ}, and called {\em abstract one-skeleta} there. However in the subsequent literature, they have been referred to as abstract GKM graphs, so we do so as well. 

 \begin{definition}[{\bf Abstract GKM$_k$ graph}] \label{defn:GKMkgraph}
Let $k\geq 2$. Then a {\em GKM$_k$ graph}, $(\Gamma,\alpha,\nabla)$, or simply $(\Gamma,\alpha)$ when $k\geq 3$, 
consists of a 
graph  $\Gamma$, an axial function $\alpha: E(\Gamma)\rightarrow {\frak{t}}_{\mathbb{Q}}^*/\{\pm 1\}$,  such that
\begin{enumerate}
\item the underlying graph $\Gamma$ is $n$-valent;
\item for any $e\in E(\Gamma)$, $\alpha(\bar e) = \alpha(e)$; and 
\item for any $v\in V(\Gamma)$ and any set of $k$ distinct edges $e_1,\ldots,e_k\in E_v$, the elements $\alpha(e_1),\ldots,\alpha(e_k)$ are linearly independent;
\item for any $v\in V(\Gamma)$ and any pair of distinct edges $e,f\in E_v$, we have that 
$$\alpha(\nabla_e f)=\pm \alpha(f) + c_{e, f}\alpha(e),$$
for some constant $c = c_{e, f}\in \mathbb{Z}$, depending on $e$ and $f$.

\end{enumerate}
\end{definition}

\begin{remark}\label{canonical} The abstract GKM$_k$ graph for $k\geq 3$ consists of a triple, but since the connection is canonical for $k\geq 3$ \cite{GZ}, there is no need to list it in this case.
 \end{remark}

We  also define a variant of the abstract GKM$_k$ graph which is signed, motivated by the fact that one often assumes that a GKM manifold admits an invariant almost complex structure.
\begin{definition}[{\bf Abstract Signed GKM$_k$ graph}] \label{remark2.18} Let $k\geq 2$. Then a {\em signed GKM$_k$ graph} is a GKM$_k$ graph where $\pm$ is replaced by $+$. That, is, the following modifications are made to Definition \ref{defn:GKMkgraph}.
 \begin{enumerate} 
 \item The axial function now takes values in the $R$-module $V$, that is, $\alpha: E(\Gamma)\rightarrow V$. 
 \item In Part (4)  is modified to be $$\alpha(\nabla_e f)= \alpha(f) + c_{e, f}\alpha(e).$$

\end{enumerate}
 \end{definition}

In order to motivate the previous definitions, we now describe the geometric graph obtained from a GKM$_k$ manifold, $M^{2n}$.
\begin{definition}[{\bf Geometric GKM$_k$ graph}]\label{geometric} Let $M^{2n}$ be a GKM$_k$ manifold. Then the {\em geometric GKM$_k$ graph of $M$} is an abstract GKM$_k$ graph,  
$(\Gamma_M, \alpha_M, \nabla_M)$ or simply $(\Gamma_M, \alpha_M)$ for $k\geq 3$,   
where $\Gamma_M$ and    $\alpha_M$ 
are defined as follows. 
\begin{enumerate}
\item $\Gamma_M$  is the quotient by the torus action of the $1$-skeleton, $M_1/T$, considered as a graph.

\item The axial function $\alpha_M: E(\Gamma_M)\rightarrow {\frak{t}}_{\mathbb{Q}}^*/\{\pm 1\}$ is defined on each edge, $e$,  of $\Gamma_M$, as the corresponding weight 
of the isotropy representation, considered as an element of 
${\frak{t}}_{\mathbb{Q}}^*/\{\pm 1\}$. 
\end{enumerate}
\end{definition}
\begin{remark}
With this definition 
$V(\Gamma_M)$ corresponds to the (isolated) fixed points of the torus action and 
$E(\Gamma_M)$ corresponds to the $2$-spheres fixed by a codimension one subtorus of $T$, containing two isolated fixed points of $T$. 
\end{remark}

Equivalently,  given a GKM$_k$ action of a torus $T$ on an orientable manifold $M^{2n}$, the geometric graph of this action is constructed as follows: we have one vertex for each fixed point, every invariant two-sphere contains exactly two fixed points, and we associate to it one edge connecting the corresponding vertices. The axial function associates to each edge, that is, invariant two-sphere, the corresponding weight of the isotropy representation.

 \begin{remark}  
Given a GKM$_k$-manifold $M^{2n}$, for every choice of $l\leq k-1$ edges $e_1,\ldots,e_l$ at a vertex $p$, we can define a $2l$-dimensional $T$-invariant submanifold $N^{2l}$ of $M$, that is itself GKM$_l$ with the induced torus action. In fact, we take ${\mathfrak{h}} := \bigcap_{i=1}^l \ker \alpha_M(e_i)$, consider the subtorus $H\subset T$ with Lie algebra $\mathfrak h$, and let $N$ be the connected component of $M^H$ containing the fixed point $p$. The GKM$_k$ condition implies that $N$ is $2l$-dimensional, as its tangent space at $p$ is precisely the sum of those weight spaces whose weights vanish on $\mathfrak h$. We then see that in the above definition 
of $\Gamma_M$,  the graph of the GKM$_k$ manifold $M$, the image of  the $1$-skeleton of $N$ is  a subgraph $\Gamma_N\subset \Gamma_M$, and we call this subgraph a {\em face}. In the special case when $M^{2n}$ is a quasitoric manifold, then $M/T$ is an $n$-dimensional simple polytope, $P^n$. In this special case, $\Gamma_M$ is the $1$-skeleton of $P^n$ and each face of $P^n$ corresponds to some $N^{2l}/T$, whose $1$-skeleton is $\Gamma_N$.
 \end{remark}

We now explicitly describe the connection, $\nabla_M$, in the GKM$_3$ case, as it is easier to describe than in the GKM case, and moreover we are only concerned with the GKM$_3$ case in this article.
 In this setting, given any two distinct edges $e,f$ emanating from a vertex $v$, there is a unique $2$-dimensional face, $F$, 
containing $e$ and $f$. Let $e'\neq \bar{e}$ be the unique edge in  $F$,  such that $i(e')=t(e)$.
Setting  $(\nabla_M)_e f=e'$, it follows from \cite{GW2} that $(\Gamma_M, \alpha_M, \nabla_M)$ satisfies the conditions for an abstract GKM$_3$ graph.
In particular, we see that the connection allows us to slide edges along edges inside the corresponding two-dimensional face of the graph. 

This leads us to define the concept of an {\em abstract face} of an abstract GKM$_k$ graph.  Let $(\Gamma, \alpha, \nabla)$ be an abstract  GKM$_k$ graph, then we say that $\Gamma'$ is an {\em $l$-dimensional face} of $(\Gamma, \alpha, \nabla)$ if it is an $l$-valent subgraph, invariant under $\nabla$.

We are now in a position to define the equivariant cohomology of an abstract  GKM graph, which was first done in Section 1.7 of \cite{GZ} and was denoted simply by $H(\Gamma,\alpha)$. Since we also define the cohomology of an abstract  GKM graph  below in Definition \ref{defn:cohomgraph}, to avoid confusion, we denote the equivariant cohomology of an abstract  GKM graph by $H_T(\Gamma,\alpha)$. 

\begin{definition} [{\bf Equivariant Cohomology of an abstract GKM Graph}]\label{eqco}
The \emph{equivariant cohomology} of a GKM graph $(\Gamma,\alpha,\nabla)$ is defined as 

$$
H_T^* (\Gamma,\alpha) = \{(f_v)_{v\in V(\Gamma)}\in \bigoplus_{v\in V(\Gamma)} S({\mathfrak{t}}_{\mathbb{Q}}^*)\mid \alpha(e)\text{ divides }f_{i(e)}-f_{t(e)}\text{ for all } e\in E(\Gamma)\},
$$
where the generators of $S({\mathfrak{t}}_{\mathbb{Q}}^*)$ are assigned degree $2$. It is naturally an $S({\mathfrak{t}}_{\mathbb{Q}}^*)$-algebra.
\end{definition}

We now recall Theorem 1.2.2, also known as the GKM Theorem, in \cite{GKM} here.
\begin{theorem}[{\bf GKM Theorem}]\cite{GKM}\label{thm:GKM}
For a GKM action of a torus $T$ on $M$, with GKM graph $\Gamma_M$, the injection $H^*_T(M)\to H^*_T(M^T) = \bigoplus_{p\in M^T}S({\mathfrak{t}}_{\mathbb{Q}}^*)$ has as image exactly $H^*_T(\Gamma_M,\alpha_M)$. Thus, $H^*_T(M)\cong H^*_T(\Gamma_M,\alpha_M)$ as $S({\mathfrak{t}}_{\mathbb{Q}}^*)$-algebras.
\end{theorem}
Motivated by Proposition \ref{Ecohom} we define the cohomology of an abstract GKM graph as follows.
\begin{definition} [{\bf Cohomology of an abstract GKM Graph}] \label{defn:cohomgraph} 
The \emph{cohomology} of a GKM graph $(\Gamma,\alpha,\nabla)$ is defined as
\[
H^*(\Gamma,\alpha) = \frac{H^*_T(\Gamma,\alpha)}{S^+({\mathfrak{t}}_{\mathbb{Q}}^*)\cdot H^*_T(\Gamma,\alpha)}.
\]
\end{definition}
\begin{remark}\label{MGamma} Thus, for a GKM action of a torus $T$ on $M$, we have $H^*(M)\cong H^*(\Gamma_M)$ by Theorem \ref{thm:GKM} and Proposition \ref{Ecohom}. 
\end{remark}
\subsection{Results for GKM$_k$ manifolds in positive and non-negative curvature}
We state here two results for GKM$_k$ manifolds of positive and non-negative sectional curvature that we need for the proofs of Theorem \ref{2} and \ref{3}.
The first is a classification result for positively curved GKM$_3$ manifolds from \cite{GW1}.
\begin{theorem}\label{GKM3+}\cite{GW1}
Let $M$ be a closed, positively curved, orientable GKM$_3$ Riemannian manifold. 
Then $M$ has the real cohomology ring of a compact rank one symmetric space.
\end{theorem}
\noindent The second  is a classification result  for non-negatively curved GKM$_4$ manifolds from \cite{GW2}.

\begin{theorem}\label{GKM40}\cite{GW2} Let $M$ be a non-negatively curved GKM$_4$ manifold. Then $$H^*(M)\cong H^*(\widetilde{M}/G),$$
where $\widetilde{M}$ is a simply-connected, non-negatively curved torus manifold and $G$ is a finite group acting isometrically on $\widetilde{M}$.
\end{theorem}

\subsection{Odd-dimensional GKM theory}\label{sec:odddimGKM}
GKM theory was generalized to torus actions on odd-dimensional manifolds with one-dimensional fixed point set in \cite{H}.  The odd GKM condition is an odd-dimensional analogue of the even-dimensional GKM condition. We describe the corresponding theory here.

\begin{definition}[{\bf Odd GKM Torus Action and Manifold}]\label{oddGKM}
We say that the action of a torus $T$ on an orientable, compact, connected manifold $M^{2n+1}$ is \emph{odd GKM} if
\begin{enumerate}
\item the fixed point set $M^T$ of the action is a finite union of 
circles;
\item for every $p\in M^T$, the weights $\alpha_{i, p}\in {{\mathfrak{t}}_{\mathbb{Q}}}^*/\{\pm 1\}, i=1, \hdots, n,$ of the isotropy representation of $T$ on $T_pM$ are pairwise linearly independent; and
\item it is equivariantly formal.
\end{enumerate}
\end{definition}

\begin{remark}
Condition $3$ is equivalent to requiring the \emph{$1$-skeleton} 
$M_1 = \{p\in M\mid \dim\, (T \cdot p) \leq 1\}$ to be a finite union of $3$-dimensional $T$-invariant submanifolds.

Recall that for GKM manifolds, Condition $2$ of   Definition \ref{evenGKM} implies that up to diffeomorphism, there is only one $2$-dimensional $T$-invariant submanifold, $S^2$,  and  that $S^2\cap M^T$ consists of exactly $2$ isolated points.
In contrast, here there are an infinite number of possible $T$-invariant $3$-manifolds that could occur, see Section $4$ of \cite{H} for a complete classification of such $3$-manifolds. 

\end{remark}

Note that in \cite{H}, it is neither assumed that the action is equivariantly formal, nor that $M$ is orientable. We include these conditions because we do so in the even-dimensional setting, and as mentioned earlier, both are often part of the definition of a GKM action.

\begin{remark} One should note that in \cite{GNT} a different generalization of GKM theory to odd dimensions was introduced for so-called Cohen-Macaulay actions. The GKM-type actions considered there are more general than those of \cite{H}, as they do not necessarily have fixed points. That is, using the stratification induced by the $M_k$ skeleta, one only has that $M_l\neq\emptyset$ for some $l\geq 0$, rather than $M_0=M^T\neq\emptyset$.  On the other hand, the definition from \cite{GNT} is also more restrictive in terms of the stratification of the $k$-skeleta. 
Namely, given $N$, a connected component of $M_{l+1}\setminus M_l$, with $M_l\neq \emptyset$,  $N$ contains exactly two components of $M_l$. By contrast, in the definition given in \cite{H}, the number of such components is greater than or equal to $1$.
\end{remark}

As in classical GKM theory one can associate to an odd GKM manifold a  \emph{geometric odd GKM graph}, from which one can, for equivariantly formal actions, compute the equivariant, as well as the ordinary rational cohomology of the manifold.

 To encode the structure of the $1$-skeleton in a graph,  two types of vertices are defined in \cite{H}. We will decorate our graphs with a bar to distinguish them from even-dimensional GKM graphs.

\begin{definition}[{\bf Vertex types}]\label{vertex}
The graph, $\bar{\Gamma}_M$, of an odd GKM manifold $M$ has two types of vertices:
\begin{enumerate}
\item  One \emph{circle} for each circle in the fixed point set; and 
\item One \emph{square} for each invariant three-dimensional submanifold in $M_1$, the $1$-skeleton of $M$. 
\end{enumerate}
We denote the set of circles by $V_{\scaleto{{\scaleto{\circ}{4pt}}}{4pt}}$ and the set of squares by $V_{\scaleto{\square}{4pt}}$. 

\end{definition}

We also have restrictions on how edges are formed and distinguish between two particular types.\begin{definition} [{\bf Edge and Edge types}]  \label{edge}
We connect a circle to a square by an \emph{edge} if the fixed circle is contained in the corresponding three-dimensional submanifold. Further, at any circle in the graph, $\bar{\Gamma}_M$, we distinguish between the following edge types:

\begin{enumerate}
\item a \emph{floating edge}, that is, an edge connecting to a square of  valence $1$,  and 
\item a \emph{grounded edge}, that is, an edge connecting to a square of valence $\geq 2$. 
\end{enumerate}

\end{definition}

We further distinguish the following important subsets of $V_{\scaleto{{\scaleto{\circ}{4pt}}}{4pt}}$ and $V_{\scaleto{\square}{4pt}}$, namely, we denote by $V_{\scaleto{{\scaleto{\circ}{4pt}}}{4pt}}(s)$ the set of circles connected to $s\in V_{\scaleto{\square}{4pt}}$, and  by $V_{\scaleto{\square}{4pt}}(c)$ the set of squares connected to $c\in V_{\scaleto{{\scaleto{\circ}{4pt}}}{4pt}}$. 

In analogy with how weights are assigned to  the  GKM$_k$ graph, we now assign weights to square vertices rather than edges.

\begin{definition}[{\bf Weight function}]\label{weight} The {\em weight function} $\alpha_M:V_{\scaleto{\square}{4pt}}\to {\mathfrak{t}}_{\mathbb{Q}}^*/\{\pm 1\}$ assigns to each square $s$ of  $\bar{\Gamma}_M$,  a  weight, $\alpha_M(s)$, which is 
the weight 
of the isotropy representation at any fixed circle in the three-dimensional submanifold corresponding to $s$, considered as an element of 
${\frak{t}}_{\mathbb{Q}}^*/\{\pm 1\}$.

\end{definition}

\noindent Note that by definition, any edge connects a circle to a square. 

We can also introduce a notion of connection on such a graph, as in the even-dimensional setting. The only difference is that we do not specify a single edge along which we transport, but two circles in the same three-dimensional component of the $1$-skeleton.
\begin{definition}\label{GKMgraph}
A  {\em connection on the  graph, $\bar{\Gamma}_M$, of an odd-dimensional GKM manifold} $M$ is a collection of maps $(\bar{\nabla}_M)_{c_1,c_2,s_0}:V_{\scaleto{\square}{4pt}}(c_1)\to V_{\scaleto{\square}{4pt}}(c_2)$, for every $s_0\in V_{\scaleto{\square}{4pt}}$ and $c_1,c_2\in V_{\scaleto{{\scaleto{\circ}{4pt}}}{4pt}}(s_0)$,  satisfying the following conditions:
\begin{enumerate}
\item $(\bar{\nabla}_M)_{c_1,c_2,s_0}(s_0) = s_0$
\item $(\bar{\nabla}_M)_{c_2,c_1,s_0} = (\bar{\nabla}_M)_{c_1,c_2,s_0}^{-1}$
\item For every $s\in V_{\scaleto{\square}{4pt}}(c_1)$, there exists a constant $c\in {\mathbb{Z}}$ such that $$\bar{\alpha}_M((\bar{\nabla}_M)_{c_1,c_2,s_0}(s)) = \pm \bar{\alpha}_M(s) + c \, \bar{\alpha}_M(s_0).$$
\end{enumerate}
\end{definition}

The following proposition guarantees the existence of a connection on the  graph of every odd-dimensional GKM manifold. 
The proof  is completely analogous to the proof of Proposition 2.3  in \cite{GW2}, or the proof on page 5 of \cite{GZ}.

\begin{proposition}\label{lem:odddimconnection}
There exists a  connection on the   graph, $\bar{\Gamma}_M$, of every odd-dimensional GKM manifold.
\end{proposition}

 With these notions, we now define a {\em  geometric odd GKM graph}.

\begin{definition}[{\bf Geometric Odd GKM graph}] Let $M^{2n+1}$ be an odd GKM manifold. We define a {\em geometric odd GKM graph}, $(\bar{\Gamma}_M, \bar{\alpha}_M, \bar{\nabla}_M)$,  where $\bar{\Gamma}_M$ is the graph obtained from the $1$-skeleton of $M$, with vertices, edges and weight function, $\bar{\alpha}_M$, and a connection, $ \bar{\nabla}_M$, as described above.
\end{definition}

\begin{remark} For odd GKM graphs whose squares have valence $1$ or $2$ only, we will use the following notational shortcut. Namely, we will denote each $2$-valent square $s\in V_{\scaleto{\square}{4pt}}$, by $s_{ij}$ where $c_i,c_j\in V_{\scaleto{{\scaleto{\circ}{4pt}}}{4pt}}(s)$ are the unique circle vertices connecting to $s$. . 
In particular, in analogy with the orientation assigned to edges in the even-dimensional case, this allows us to assign an orientation to a $2$-valent square, namely we let  $\bar{s}_{ji}=s_{ij}$. 
\end{remark}

\begin{example}\label{ex:pinwheel} The  geometric odd GKM graph of a $(2n+1)$-dimensional sphere $S^{2n+1}\subset \ccc^{n+1}$ with the standard $T^n$-action induced by the standard representation on $n$ of the $n+1$ summands is a pinwheel with $n$ edges terminating in squares, corresponding to fixed $3$-spheres, as follows: 

\hskip .2cm
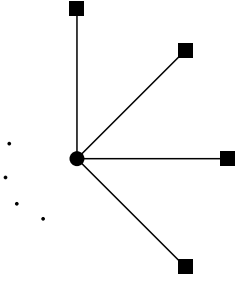
\begin{center}
\begin{tikzpicture}[square/.style={regular polygon,regular polygon sides=4}]

\draw[line width=.2mm] (0,0) -- (0,2);
\draw[line width=.2mm] (0,0) -- (1.44,1.44);
\draw[line width=.2mm] (0,0) -- (2,0);
\draw[line width=.2mm] (0,0) -- (1.44,-1.44);
\node at (0,0)[circle,fill,inner sep=2pt]{};
\node at (0,2)[square,fill,inner sep=2pt]{};
\node at (1.44,1.44)[square,fill,inner sep=2pt]{};
\node at (2,0)[square,fill,inner sep=2pt]{};
\node at (1.44,-1.44)[square,fill,inner sep=2pt]{};

\node at (-.9,.2)[circle, fill, inner sep = .5pt]{};
\node at (-0.95,-.25)[circle, fill, inner sep = 0.5pt]{};
\node at (-.8,-0.6)[circle, fill, inner sep = 0.5pt]{};
\node at (-.45,-0.8)[circle, fill, inner sep = 0.5pt]{};

\end{tikzpicture}
\end{center}

\end{example}
In the following proposition we collect a few properties of odd GKM graphs.

\begin{proposition}\label{props} Let $M^{2n+1}$ be an odd GKM manifold, then the following hold:
\begin{enumerate}
\item The  geometric odd GKM graph is connected
\item Each circle in the geometric odd GKM graph has valence $n$.

\item If the total Betti number of $M^{2n+1}$ is $2m$, then there are exactly $m$ circles in the graph. Moreover, 
each square in the  odd GKM graph has valence bounded between $1$ and $m$. 
\end{enumerate}
\end{proposition}
\begin{proof} We include a proof for the sake of completeness, noting that the proof follows along the same lines as in the even-dimensional case.  To prove Part 1, we note 
the Chang-Skjelbred Lemma \ref{CSLemma} states that for an equivariantly formal action, the image of the injective map $H^*_T(M)\to H^*_T(M^T)$ is the same as the image of the map $H^*_T(M_1) \to H^*_T(M^T)$. As $M$ is connected, it follows that the image of $H^0(M_1) \cong H^0_T(M_1) \to H^0_T(M^T)\cong H^0(M^T)$ is one-dimensional, which implies that $M_1$ is connected. This is equivalent to the GKM graph being connected.

To prove Part 2,  recall that any edge emanating from a circle corresponds to a weight of the isotropy representation at that circle. Because the codimension of this circle is $2n$, and $T$ acts on the normal space of this circle without fixed vectors, there are precisely $n$ such weights.

To prove Part 3, it follows by Proposition \ref{eformal} that the equivariant formality of the action is equivalent to the equality of total Betti numbers $\dim\, H^*(M) = \dim\, H^*(M^T)$. Because any circle in $M^T$ contributes $2$ to the total Betti number of $M^T$, and 
since any square must contain a circle fixed by $T$ and can contain at most $m$ circles, the result follows.
\end{proof}

We are now in a position to define an odd GKM$_k$ manifold.
\begin{definition}[{\bf Odd GKM$_k$ manifolds}]
An odd-dimensional GKM manifold is called {\em odd GKM$_k$}, for $k\geq 2$, if the following hold.
\begin{enumerate} 
\item $M$ is odd-dimensional GKM, and;
\item At any fixed circle, any $k$ weights of the isotropy representation are linearly independent.
\end{enumerate}
\end{definition} 
Thus, odd GKM manifolds are the same as odd GKM$_2$ manifolds. In the same way that  a geometric odd GKM graph is associated to an odd GKM manifold, we obtain an odd geometric GKM$_k$ graph from an odd GKM$_k$ manifold. 
 Note that for geometric odd GKM$_3$ graphs, for every $s_0\in V_{\scaleto{\square}{4pt}}$, given $c_1,c_2\in V_{\scaleto{{\scaleto{\circ}{4pt}}}{4pt}}(s_0)$, the condition $\bar{\alpha}_M((\bar{\nabla}_M)_{c_1,c_2,s_0}(s)) = \pm \bar{\alpha}_M(s) + c \bar{\alpha}_M(s_0)$ alone uniquely determines the square $(\bar{\nabla}_M)_{c_1,c_2,s_0}(s)$, for all $s\in V_{\scaleto{\square}{4pt}}(c_1)$, that is, the connection is unique.
 Since the connection for a geometric odd GKM$_k$ graph is canonical for $k\geq 3$, we will denote 
such graphs simply by  $(\bar{\Gamma}_M, \bar{\alpha}_M)$.

For a GKM$_k$ manifold $M$, and any $k-1$ weights at a fixed circle, there is a unique ($2k-1$)-dimensional submanifold fixed by a codimension $k-1$ subtorus generated by the intersection of the kernels of the $k-1$ weights. We will denote this submanifold by $N_{T}^{2k-1}$. As in the even-dimensional case, we make the following definition of a face.
\begin{definition}[{\bf Face}]
We call the subgraph of the GKM$_k$ graph of $M$ corresponding to $N_{T}^{2k-1}$ a \emph{($k-1$)-face} of the graph.
\end{definition}

Theorem 4.6 in \cite{H} tells us how the GKM graph encodes the equivariant cohomology, whose relevant content we recall here. 
We choose an orientation on every component of $M^T$, which then allows us to identify its cohomology canonically with $H^*(S^1) = \qqq[\theta]/(\theta^2)$. The inclusion $M^T\to M$ induces an injection
\[
H^*_T(M) \longrightarrow H^*_T(M^T) =\bigoplus_{C\in V_{\scaleto{{\scaleto{\circ}{4pt}}}{4pt}}} S({\mathfrak{t}}_{\mathbb{Q}}^*)\otimes H^*(S^1),
\]
and the image of this map is described by the following divisibility relations.  For any $s\in V_{{\scaleto{\square}{4pt}}}$, let $N_s$  be the $3$-dimensional connected submanifold fixed by the subtorus with Lie algebra $\ker \alpha(s)$. Then, for  $c_1,\ldots,c_l\in V_{\scaleto{{\scaleto{\circ}{4pt}}}{4pt}}$, the circles contained in $N_s$, we have
 $$
(P_c+Q_c\theta)_{c\in V_{\scaleto{{\scaleto{\circ}{4pt}}}{4pt}}}\in \bigoplus_{c\in V_{\scaleto{{\scaleto{\circ}{4pt}}}{4pt}}} S({\mathfrak{t}}_{\mathbb{Q}}^*)\otimes H^*(S^1),
$$
where $P_c, Q_c\in S({\mathfrak{t}}_{\mathbb{Q}}^*)$, satisfies
 \begin{equation}\label{D2}
P_{c_1}\equiv \cdots \equiv P_{c_l}\, \text{ mod } \alpha(s)
\end{equation}
and
 \begin{equation}\label{D3}
\sum_{i=1}^l \pm Q_{c_i} \equiv 0\, \text{ mod }\alpha(s).
\end{equation}
Here, the $\pm$ signs in the sum are determined as follows.  Recall that for a closed manifold $M$, the fixed point sets of torus actions are closed submanifolds that are orientable if $M$ is. Thus $N_s$ is orientable, and so the orbit space $N_s/T$ is orientable as a topological manifold (with boundary), as well.  The circles $c_i$ are boundary components of $N_s/T$, and if 
the pre-chosen orientation on each $c_i$ coincides with induced boundary orientation, with respect to any orientation of $N_s/T$, then 
the sign of $Q_{c_i}$ is $+$, and if not, then its sign is $-$.

\begin{remark}
It is not possible in general to consistently orient all components of $M^T$ in such a way that for all $N_s$ we find an orientation on $N_s/T$  with the property that the circles in $N_s$ carry the induced boundary orientation. Consider, for example, $S^1\times \ccc P^2$ with the standard $T^2$ product action which is trivial on the first factor. \end{remark}

\subsection{Geometric results in the presence of a lower curvature bound}

We now recall some general results about $G$-manifolds with non-negative curvature which we use throughout.

We recall the classification of closed, non-negatively curved $T^1$-fixed point homogeneous manifolds due to Galaz-Garc\'ia \cite{GG}.
\begin{theorem}\label{GG}\cite{GG} Let $M^3$ be a closed, non-negatively curved $T^1$-fixed point homogeneous Riemannian manifold. Then $M$ is diffeomorphic to one of 
$S^3$, $L_{p, q}$, $S^2\times S^1$, $S^2\tilde{\times} S^1$, the non-trivial $S^2$-bundle over $S^1$, $\rrr P^2\times S^1$, or $\rrr P^3\# \rrr P^3$. 

Moreover, an analysis of the isometric circle action yields the following.

\begin{enumerate}
\item If $M^3$ has  total Betti number equal to $2$, the isometric circle action fixes one circle; and 
\item If $M^3$ has  total Betti number equal to $4$, the isometric circle action fixes two circles.
\end{enumerate}
\end{theorem}
\begin{remark} We make the following two observations.
\begin{enumerate}
\item By Proposition \ref{eformal}, it follows that non-negatively curved $S^1$-fixed point homogeneous $3$-manifolds are equivariantly formal.

\item
The only orientable manifold on this list that is not a rational cohomology sphere is $S^2\times S^1$. Moreover, $S^2\times S^1$ is the only  manifold on this list with total Betti number equal to $4$.
\end{enumerate}
\end{remark}

The  following theorem by Spindeler,  \cite{Spi}, gives a characterization of non-negatively curved $G$-fixed point homogeneous manifolds.

\begin{theorem}\label{Spindeler} \cite{Spi} Assume that $G$ acts fixed point homogeneously  on a closed, non-negatively curved Riemannian manifold $M$. Let $F$ be a fixed point component of maximal dimension. Then there exists a smooth submanifold $N$ of $M$, without boundary, such that $M$ is diffeomorphic to the normal disk bundles $D(F)$ and $D(N)$ of $F$ and $N$ glued together along their common boundaries
\begin{equation*}\label{Decomposition}
M  = D(F) \cup_{\partial} D(N).
\end{equation*}
Further, $N$ is $G$-invariant and all points of $M\setminus \{F\cup N\}$ belong to principal $G$-orbits. 
\end{theorem}

\begin{remark}\label{Hinvariant} In fact $N$ is actually $\mathrm{Isom}_F(M)$-invariant, where  $\mathrm{Isom}_F(M)$ is the subgroup of isometries of $M$ leaving $F$ invariant, by Lemma 3.30 in \cite{Spi}.
\end{remark}
Finally, we recall the following Splitting Theorem due to Cheeger and Gromoll \cite{CG}.
\begin{theorem}\label{CG}\cite{CG}  Let $M$ be a compact manifold of non-negative Ricci curvature. Then $\pi_1(M)$ contains a finite normal subgroup $\Psi$ such that $\pi_1(M)/\Psi$ is a finite
group extended by $\zzz^k$, and $\widetilde{M}$, the universal covering of $M$, splits isometrically as $\overline{M} \times \rrr^k$, where $\overline{M}$ is compact.
\end{theorem}

\section{Closed, non-negatively curved $5$-dimensional GKM manifolds} \label{s3}

Observe that a closed $(2n+1)$-dimensional, equivariantly formal manifold admitting an effective, isometric $T^n$ action is GKM. 
In this section we prove Theorem \ref{m5}, which classifies the universal covers of closed, non-negatively curved, $5$-dimensional, equivariantly formal manifolds admitting an effective, isometric $T^2$ action and hence the universal covers of closed, non-negatively curved, $5$-dimensional GKM manifolds.  Theorem \ref{m5} will   facilitate the proof of Theorem \ref{graphs}, where we classify the geometric graphs of closed, non-negatively curved, $5$-dimensional GKM manifolds. 

We also note that as $M^T\neq \emptyset$ by equivariant formality,  some $S^1\subset T^2$ fixes a codimension $2$ submanifold in $M^5$, that is, 
the $T^2$-action on $M$ is  $S^1$-fixed point homogeneous.  Recall that by Theorem \ref{Spindeler} if $M^5$ admits an isometric $T^2$-action that is $S^1$-fixed point homogeneous, then we may decompose $M^5$ as a union of disk bundles, that is 
\begin{equation}\label{decomp}
M^5=D^2(F^3)\cup_ED(N),
\end{equation} where $F^3$ is the codimension two fixed point set of some circle subgroup of $T$, $E$ is the common boundary of the disk bundles, and by Remark \ref{Hinvariant}, $N$ is a $T^2$-invariant submanifold. Moreover, $F^3$ is itself $S^1$-fixed point homogeneous by Proposition \ref{ef}.

\begin{theorem}\label{m5} Let $M^5$ be a closed, non-negatively curved, equivariantly formal $5$-dimensional   manifold  admitting an isometric $T^2$-action.  
Then $\rk(H_1(M^5; \zzz))\leq 1$ and
 we may classify the corresponding universal cover, $\widetilde{M}^5$, as follows.
\begin{enumerate}
\item For $\rk(H_1(M^5;\zzz)) = 0$,  $\widetilde{M}^5$ is diffeomorphic to one of  $S^5$, $S^3\times S^2$, or  $S^3\tilde{\times} S^2$, the non-trivial $S^3$ bundle over $S^2$; 
\item For $\rk(H_1(M^5;\zzz))=1$, $\widetilde{M}^5$ is diffeomorphic to $\rrr\times M^4$, where $M^4$ is one of $S^4$, $\ccc P^2$, $S^2\times S^2$, or $\ccc P^2\# \pm \ccc P^2$. 
\end{enumerate}
\end{theorem}

Combining the homeomorphism classification due to Rong \cite{R} and the diffeomorphism classification results of Smale \cite{Sm} and Barden \cite{Ba}, we have the following theorem for the positive curvature case.

\begin{theorem}\label{Rong}\cite{R} Let $M^5$ be a closed, simply-connected, positively curved $5$-dimensional  manifold admitting an isometric $T^2$-action. Then  $M^5$ is diffeomorphic to $S^5$.
\end{theorem}

Since positively curved manifolds have finite fundamental group, the following corollary is immediate, 
allowing us to classify the universal covers of positively curved $5$-dimensional GKM manifolds as follows.

\begin{corollary} Let $M^5$ be a closed, positively curved, equivariantly formal  $5$-dimensional  manifold admitting an isometric $T^2$-action. Then the universal cover of $M^5$ is diffeomorphic to $S^5$.\end{corollary}

In order to prove Theorem \ref{m5}, 
we first need to prove Proposition \ref{Betti} and Lemmas \ref{even} and \ref{h1.1}, which follow. 

\begin{proposition}\label{Betti} Let $M^5$ be a  closed, orientable, equivariantly formal, non-negatively curved $5$-dimensional Riemannian manifold with an isometric $T^2$-action. 
Then the following hold for $b(M^5)$,  the total Betti number of $M^5$.
\begin{enumerate}
\item $2 \le b(M^5) \le 8$; and
\item
If $\,6 \le b(M^5) \le 8$, then in the decomposition of $M^5$ in Display \ref{decomp}, $F^3$ is diffeomorphic to $S^2\times S^1$ . 
\end{enumerate}
\end{proposition}

Before we begin the proof, we need the following result concerning the dimension of the submanifold at maximal distance from the codimension two fixed point set in $M^{2n+1}$.

\begin{lemma}\label{even} Suppose $M^{2n+1}$ is an $S^1$-fixed point homogeneous closed, orientable manifold of non-negative curvature.  Let $F$ be a codimension two fixed point set component of the $S^1$-action and suppose $N$ is the submanifold  given in the disk bundle decomposition of Theorem \ref{Spindeler}. 
Then if $N\cap \Fix(M; S^1)\neq \emptyset$, $\codim(N)$ is even.
\end{lemma}
\begin{proof} Recall that 
by Theorem \ref{Spindeler} all singularities of the $S^1$-action are contained in $F$ and $N$ and $N$ is $S^1$-invariant.
Let $N\cap \Fix(M; S^1)\neq \emptyset$.  Then any connected component  $A$ of $\Fix(M; S^1)$ that is not contained in $F$  must be contained in $N$, and is of  even codimension in $M$.
Since $N$ is $S^1$-invariant, $A$ is  also of even codimension in $N$ and the result follows. 
\end{proof}

We are now in a position to prove Proposition \ref{Betti}.

 \begin{proof}[{\bf Proof of Proposition \ref{Betti}}]
 Recall that the hypothesis of equivariant formality guarantees that the $T^2$-action on $M$ is  $S^1$-fixed point homogeneous and the manifold decomposes as in Display \ref{decomp}.
Moreover, by Proposition \ref{eformal} the total Betti number of $\Fix(M^5; S^1)$ equals that of $M^5$. Suppose first that $N$ does not contain any $S^1$-fixed points, that is, $\Fix(M^5; S^1)=F^3$. Since $F^3$ is itself non-negatively curved and $S^1$-fixed point homogeneous with respect to some other subcircle of $T^2$,  the proposition is proven, as by Theorem \ref{GG}, the total Betti number of $F^3$ is either $2$ or $4$.

We now assume that there are $S^1$-fixed points in $N$. This implies by Lemma \ref{even} that $N$ is of dimension $1$ or $3$ only. If $N$ is $1$-dimensional, it follows from the classification of $1$-manifolds that $N=S^1$, and so, the total Betti number of $M$ is either $4$ or $6$.

Assume then that $N$ is $3$-dimensional. If $N^3$ is fixed by some circle subgroup of $T^2$, it is an orientable, totally geodesic submanifold of $M$, and thus non-negatively curved. Since the $T^2$-action is effective and equivariantly formal, by Corollary \ref{He}, this implies that $N^3$ is $S^1$-fixed point homogeneous for some $S^1\subset T^2$ and so $N^3$ is one of the manifolds classified in Theorem \ref{GG}. Therefore, the total Betti number of $N^3$ is $2$ or $4$. Then by Proposition \ref{eformal}, the total Betti number of the $S^1$-fixed point set in $N^3$ is also bounded from above by $4$. It follows that the total Betti number of $M$ is bounded by $8$, and if $F^3$ is a rational cohomology sphere, then it is bounded by $6$.

Suppose then that $N^3$ is not fixed by any circle subgroup of $T^2$.  Then
the $T^2$-action on $N^3$ is of cohomogeneity one, and since  there are $S^1$-fixed points, 
 it follows by the classification of $T^2$ cohomogeneity one  $3$-manifolds in Mostert \cite{M}, and Neumann \cite{N}, that $N^3$ must be one of $S^3$, $L_{p,q}$, $S^2\times S^1$,  $\rrr P^2 \times S^1$, or $S^2\tilde{\times} S^1$. Note that in all these cases, the total Betti number of $N^3$ is bounded between $2$ and $4$. Thus, it follows that the total Betti number of $M$ is bounded by $8$, and, again, if $F^3$ is a rational cohomology sphere, it is bounded by $6$. This proves Part ($1$).

We now prove Part ($2$).  We  assume that $\,6 \le b(M^5) \le 8$ and that $F^3$ is not diffeomorphic to $S^2 \times S^1$, to derive a contradiction.  Then by Theorem \ref{GG}, $F^3$ is a rational (co)homology sphere, and so $b(F^3)=2$.   Since $b(M^5)\leq b(F^3) + b(N)$ and we have seen that $2\leq b(N)\leq 4$, we now have  $b(M^5)\leq 6$. But then  $b(M^5)=6$ and hence $b(N)=4$.
So $N$ must be $S^2\times S^1$. 
Since $E$ is the total space of the circle bundle inside the disc bundle $D^2(F^3)\to F^3$,
it follows from the Gysin sequence  (see Switzer \cite{Sw}) that $E$
has the same rational homology as $S^1\times S^3$. Consider now the Mayer-Vietoris sequence of the decomposition \eqref{decomp}
\[H_5(M^5)\longrightarrow \qqq\longrightarrow 0 \longrightarrow H_4(M^5) \longrightarrow \qqq \longrightarrow \qqq^2 \longrightarrow H_3(M^5) \]\[\longrightarrow 0 \longrightarrow \qqq \longrightarrow H_2(M^5) \longrightarrow \qqq  \longrightarrow \qqq\longrightarrow H_1(M^5) \longrightarrow 0.
\]
Poincar\'e duality and the fact that $b(M^5)=6$ gives us that $b_2=b_3$ and hence $b_1=b_4=2-b_2$.  Exactness at $H_2(M^5)$, implies that $b_2=b_3\in\{1, 2\}$. However, it is clear that exactness is violated for $b_2=b_3=1$ or $b_2=b_3=2$. 
\end{proof}

With this information, we now obtain the following lemma.
\begin{lemma}\label{h1.1}  Let $M^5$ be a closed, orientable, equivariantly formal, non-negatively curved $5$-dimensional Riemannian manifold with an isometric $T^2$-action. 
Then $$\rk(H_1(M^5))\leq 1.$$
\end{lemma}

\begin{proof}    We assume that  $\rk(H_1(M^5))=k\geq 2$, in order to derive a contradiction. This assumption implies that
the total Betti number of $M^5$ must be greater than or equal to $6$. Recall again that $M^{5}$ decomposes as in Display \ref{decomp}, with $F^3$ the codimension one fixed point set of $S^1$. Therefore,  $F^3=S^2\times S^1$ by Proposition \ref{Betti} and so $b(F^3)=4$. Further, since all singularities of the $S^1$-action are contained in $F$ and $N$, and $b(M^T)=b(M^5)\geq 6$, by Proposition \ref{eformal},   then $M^T\cap N\neq \emptyset$. 
  Then Lemma \ref{even}  gives us that $N$ is of dimension $1$ or $3$. In fact, we will show that $\dim(N)=3$.

Recall that $E$ is a sphere bundle over both $F$ and $N$. Since $M$ is orientable, it follows from the disk bundle decomposition of $M$ that $E$ is also orientable.  We first assume $\dim(N)=1$, to derive a contradiction.  Then $E$ is an  orientable $S^3$ bundle over $S^1$, that is, $E=S^3\times S^1$,  and so $H_3(E)\cong \qqq$. However, it follows from the homology Mayer-Vietoris sequence that 
$H_4(M)\cong\qqq^k\hookrightarrow H_3(E)$, $k\geq 2$, giving us the desired contradiction.

Hence $\dim(N)=3$. Since $M^T\cap N\neq \emptyset$, the $T^2$-action  on $N$ is $S^1$-ineffective for some $S^1\subset T^2$ and $N$ is $S^1$-fixed point homogeneous for some other $S^1\subset T^2$. So $N$ is orientable and non-negatively curved.  By Theorem \ref{GG},  $N$ is one of $S^3$, $L_{p, q}$, $S^2\times S^1$ or $\rrr P^3\# \rrr P^3$.  However, 
 if $N$ is one of $S^3$, $L_{p, q}$, or $\rrr P^3\# \rrr P^3$, then by the Gysin sequence with rational coefficients,  $H_1(E)\cong \qqq$.  Poincar\'e duality then gives us that $H_3(E)\cong \qqq$, which is  
 a contradiction, as in the case of $\dim(N)=1$.

Thus both $F$ and $N$ are diffeomorphic to $S^2\times S^1$.
So $E$ is a principal $S^1$ bundle over $S^2\times S^1$ and from the homology Mayer-Vietoris sequence of $M$, we have $H_4(M^5)\cong\qqq^k\hookrightarrow H_3(E), k\geq 2$. Using the Gysin sequence with rational coefficients  
corresponding to the fibration $S^1\hookrightarrow E\rightarrow S^2\times S^1$, we see that $H_1(E)\cong \qqq$ or $\qqq^2$. Poincar\'e duality then implies that $k=2$, so $b_1=b_4=2$. 
Since $b_0=b_5=1$, it then follows from the homology Mayer-Vietoris sequence of $M^5$ that $b_2=b_3\geq 2$, which implies $b(M^5)\geq 10$. But Proposition \ref{Betti} guarantees that the total Betti number is bounded above by $8$, a contradiction.
\end{proof}

\begin{proof}[{\bf Proof of Theorem \ref{m5}}]

By Lemma \ref{h1.1},  $\rk(H_1(M^5))\leq 1$ and by Theorem \ref{CG}, $\widetilde{M}^5$ is either a closed, simply-connected, non-negatively curved manifold, or it splits isometrically as $\rrr^1\times M^4$, where $M^4$ is a closed, simply-connected, non-negatively curved $4$-manifold.
The proof of Part ($1$) then follows directly from work of Galaz-Garc\'ia and Spindeler \cite{GGSp} (cf. \cite{DES}).
The proof of Part ($2$) follows by noting that since the $T^2$-action on $M^5$ has non-empty fixed point set, we may apply Theorem \ref{lift} to lift the $T^2$-action to $\widetilde{M}^5$. By Theorem $1$ of Hano \cite{Hano}, the isometry group of $\rrr^1\times M^4$ splits as the product of the isometry groups of $\rrr^1$ and of $M^4$.  Since $T^2$ is a compact Lie group, this implies that the $T^2$-action  on the $\rrr^1$ factor is trivial and   on $M^4$ is isotropy-maximal. In particular, $M^4$ is then 
a non-negatively curved torus manifold. The classification of $4$-dimensional, non-negatively curved torus manifolds up to diffeomorphism follows from work of 
Kleiner \cite{K}, Searle and Yang \cite{SY}, and Galaz-Garc\'ia \cite{GG}. 
\end{proof}

\section{The classification of graphs corresponding to closed,  non-negatively curved $5$-dimensional odd GKM manifolds} \label{s4}

Our goal in this section is to classify the graphs corresponding to closed, non-negatively curved $5$-dimensional odd GKM manifolds. We first prove that the lower curvature bound imposes severe restrictions on odd GKM graphs.

\begin{proposition}\label{valence} Let $M^{2n+1}$ be a non-negatively curved odd GKM manifold, then each square in the graph, $\bar{\Gamma}_M$, has valence one or two.
\end{proposition}

\begin{proof} The GKM condition in combination with the assumption of non-negative curvature guarantees us that the squares correspond to $3$-dimensional components of  fixed point sets of codimension one subtori, which themselves are  $T^1$-fixed point homogeneous, non-negatively curved,  $3$-dimensional manifolds.  
The result then follows directly from Theorem \ref{GG}.
\end{proof}

The corresponding result in positive curvature follows directly from the classification of positively curved $3$-manifolds due to Hamilton \cite{Ha}. 

\begin{proposition}\label{valence1} Let $M^{2n+1}$ be a positively curved odd GKM manifold, then each square in the   graph, $\bar{\Gamma}_M$, has valence one.

\end{proposition}

In the following theorem, we obtain a classification of the graphs of closed, non-negatively curved, $5$-dimensional odd GKM manifolds. 

\begin{theorem} \label{graphs} Let $M^5$ be a non-negatively curved,  $5$-dimensional odd GKM manifold. Then its graph, $\bar{\Gamma}_{M^5}$, is one of the following possibilities, enumerated according to the total Betti number of $M^5$.

\begin{enumerate}
\item For total Betti number equal to $2$, we obtain a circle with two edges terminating in squares 

\begin{center}
\begin{tikzpicture}[square/.style={regular polygon,regular polygon sides=4}]

\draw[line width=.2mm] (-2,1) -- ++(2,0) -- ++(2,0);
\node at (-2,1)[square,fill,inner sep=2pt]{};
\node at (0,1)[circle,fill,inner sep=2pt]{};
\node at (2,1)[square,fill,inner sep=2pt]{};

\end{tikzpicture}\, .

\end{center}
\hskip .2cm

\item For total Betti number equal to $4$, we have the following two possibilities

\hskip .2cm

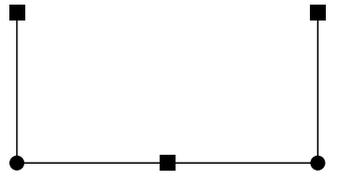
\begin{center}
\begin{tikzpicture}[square/.style={regular polygon,regular polygon sides=4}]

\draw[line width=.2mm] (-2,1) -- ++(0,-2) -- ++(2,0) -- ++(2,0) -- ++(0,2);
\node at (-2,1)[square,fill,inner sep=2pt]{};
\node at (-2,-1)[circle,fill,inner sep=2pt]{};
\node at (0,-1)[square,fill,inner sep=2pt]{};
\node at (2,-1)[circle,fill,inner sep=2pt]{};
\node at (2,1)[square,fill,inner sep=2pt]{};

\end{tikzpicture}\,\,\,\, ,

\end{center}

\hskip .2cm
and
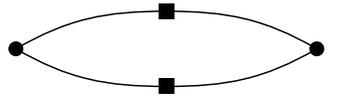
\begin{center}
\begin{tikzpicture}[square/.style={regular polygon,regular polygon sides=4}]

\draw[line width=.2mm] (-2,1) to[out=30,in=180] (0,1.5) to[out=0, in=150] (2,1) to[out=210, in=0] (0,0.5) to[out=180, in=330] (-2,1);
\node at (-2,1)[circle,fill,inner sep=2pt]{};
\node at (0,1.5)[square,fill,inner sep=2pt]{};
\node at (0,0.5)[square,fill,inner sep=2pt]{};
\node at (2,1)[circle,fill,inner sep=2pt]{};

\end{tikzpicture}\, .

\end{center}
\hskip .2cm

\item For total Betti number equal to $6$, we obtain a closed circuit  in the form of a triangle 

\hskip .2cm

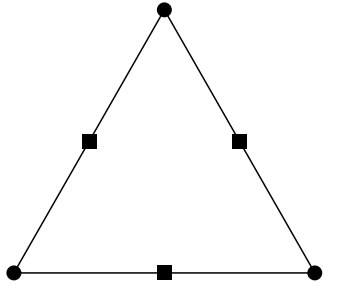
\begin{center}
\begin{tikzpicture}[square/.style={regular polygon,regular polygon sides=4}]

\draw[line width=.2mm] (-2,1) -- ++(2,0) -- ++(2,0) -- ++(-2,3.5) -- ++(-2,-3.5);
\node at (-2,1)[circle,fill,inner sep=2pt]{};
\node at (0,1)[square,fill,inner sep=2pt]{};
\node at (2,1)[circle,fill,inner sep=2pt]{};
\node at (0,4.5)[circle,fill,inner sep=2pt]{};
\node at (1,2.75)[square,fill,inner sep=2pt]{};
\node at (-1,2.75)[square,fill,inner sep=2pt]{};

\end{tikzpicture}\, .

\end{center}
\hskip .2cm

\item  For total Betti number equal to $8$, we obtain a closed circuit in the form of a quadrangle 

\hskip .2cm
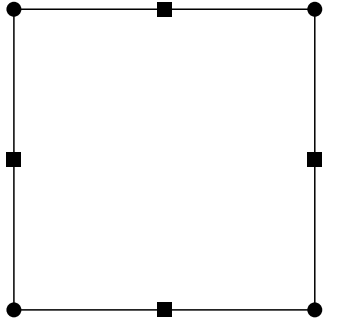
\begin{center}
\begin{tikzpicture}[square/.style={regular polygon,regular polygon sides=4}]

\draw[line width=.2mm] (-2,1) -- ++(0,-2) -- ++(2,0) -- ++(2,0) -- ++(0,2) -- ++(0,2) -- ++(-2,0) -- ++(-2,0) -- ++(0,-2);
\node at (-2,1)[square,fill,inner sep=2pt]{};
\node at (-2,-1)[circle,fill,inner sep=2pt]{};
\node at (0,-1)[square,fill,inner sep=2pt]{};
\node at (2,-1)[circle,fill,inner sep=2pt]{};
\node at (2,1)[square,fill,inner sep=2pt]{};
\node at (-2,3)[circle,fill,inner sep=2pt]{};
\node at (0,3)[square,fill,inner sep=2pt]{};
\node at (2,3)[circle,fill,inner sep=2pt]{};

\end{tikzpicture}\, .

\end{center}

\end{enumerate}
\end{theorem}

\hskip .2cm

\begin{proof} By Proposition \ref{props}, we have that  the graph is connected, 
  each circle is $2$-valent, and the number of fixed circles equals half the total Betti number. 

We showed in Proposition \ref{Betti} that the total Betti number for such $5$-manifolds is bounded between $2$ and $8$. In the case where it is $2$, $M^5$ is a rational cohomology sphere, the $T$-fixed point set is a single circle, and the graph is necessarily of the described form. If the total Betti number is $4$, the connectedness of the graph implies directly that it is of one of the two given shapes.

For the case of total Betti number $6$ or $8$,  Proposition \ref{Betti} tells us that any $F^3$ is diffeomorphic to  $S^2\times S^1$ in the decomposition of $M^5$ in Display \ref{decomp}, which 
has total Betti number $4$. This in turn implies by Proposition \ref{GG} that every square in the graph has valence $2$. The connectedness of the graph then directly implies the claim.
\end{proof}

\begin{example}
The standard examples for Theorem \ref{graphs} are the $T^2$-actions on $S^5$, $S^2\times S^3$, $S^4\times S^1$, $\ccc P^2\times S^1$ and $S^2\times S^2\times S^1$, respectively.
\end{example}

 Note that for positive curvature, using Proposition \ref{valence1}, it follows that only the first graph in Theorem \ref{graphs} occurs and we immediately obtain the following theorem.
 
 \begin{theorem}\label{posgraph}The unique graph, $\bar{\Gamma}_M$, corresponding to the 
positively curved  $5$-dimensional GKM manifolds is  a circle with two edges terminating in squares 

\hskip .2cm
\begin{center}
\begin{tikzpicture}[square/.style={regular polygon,regular polygon sides=4}]

\draw[line width=.2mm] (-2,1) -- ++(2,0) -- ++(2,0);
\node at (-2,1)[square,fill,inner sep=2pt]{};
\node at (0,1)[circle,fill,inner sep=2pt]{};
\node at (2,1)[square,fill,inner sep=2pt]{};

\end{tikzpicture}\, .
\end{center}
\hskip .2cm
\end{theorem}

\section{Proof of the Main Theorem \ref{main}} \label{s5}

In this section we  prove the Main Theorem \ref{main}, which we then use in Subsection \ref{5.2}  to prove Theorems  \ref{2}, \ref{3}, and \ref{positive}.
\subsection{Proof of the Main Theorem \ref{main}}\label{5.1}

Let $M^{2n+1}$ be a closed, non-negatively curved odd GKM$_3$ manifold. As shown in Proposition \ref{valence}, any square in the graph, $\bar{\Gamma}_M$, has valence one or two.
 
 We now show how to construct an abstract  (even-dimensional) GKM$_3$  graph, $(\Gamma, \alpha)$, from a geometric odd  GKM$_3$ graph, $(\bar{\Gamma}_M, \bar{\alpha}_M)$, for which the squares in $\bar{\Gamma}_M$ are only of valence one or two. We define $(\Gamma, \alpha)$ to be {\em the  graph obtained from the geometric odd GKM$_3$ graph} $(\bar{\Gamma}_M, \bar{\alpha}_M)$, by

 \begin{enumerate}
 \item replacing a circle by  a vertex;
\item replacing a $2$-valent square, $s$, in $2$ grounded edges by a single edge, $e$;
\item  labeling the edge $e$ obtained from $s$ by the weight of the square, $\bar{\alpha}_M(s)$, that is,  defining 
$\alpha(e):=\bar{\alpha}_M(s)$ to be the axial function; and
\item deleting all floating edges, together with their squares.
\end{enumerate}

In order to facilitate discussion of the new graph, $\Gamma$, we  denote the application, $\pi:\bar{\Gamma}_M\rightarrow \Gamma$,  of these changes to $\bar{\Gamma}_M$, respectively,  as follows:
\begin{enumerate}
\item $\pi(c)=v$, where $c\in  V_{{\scaleto{{\scaleto{\circ}{4pt}}}{4pt}}}(\bar{\Gamma}_M)$ and $v$ is its image in $\Gamma$;
\item $\pi(s)=e$, for $s\in V^2_{{\scaleto{\square}{4pt}}}$, where $s$ is the square connected to $c_1$ and $c_2$, and $e$ is its image in $\Gamma$, with $i(e)=\pi(c_1)$ and $t(e)=\pi(c_2)$. 
\item $\alpha(\pi(s)):=\bar{\alpha}_M(s)$,  
for $s\in V^2_{{\scaleto{\square}{4pt}}}$.
\end{enumerate}

Before we show that $(\Gamma, \alpha)$ is an abstract GKM$_3$ graph, we illustrate the process of obtaining these graphs in the following examples.
\begin{example}\label{examples} 
Applying this construction to the odd-dimensional graphs in Theorem \ref{graphs}, we see that we obtain the following graphs.
\begin{enumerate}
\item The image of the graph in Part $(1)$ is a vertex:

\hskip .2cm
\begin{center}
\begin{tikzpicture}[square/.style={regular polygon,regular polygon sides=4}]

\draw[line width=.2mm] (-2,1) -- ++(2,0) -- ++(2,0);
\node at (-2,1)[square,fill,inner sep=2pt]{};
\node at (0,1)[circle,fill,inner sep=2pt]{};
\node at (2,1)[square,fill,inner sep=2pt]{};

\draw[line width=.5mm, ->] (4,1) -- (6,1);

\node at (9,1)[circle, fill, inner sep = 2pt]{};
\node at (10.2,1)[circle, inner sep=0pt]{};
\end{tikzpicture}.
\end{center}
\hskip .2cm

\item The image of the graphs in Part $(2)$ are an interval and a lune, respectively:

\hskip .2cm

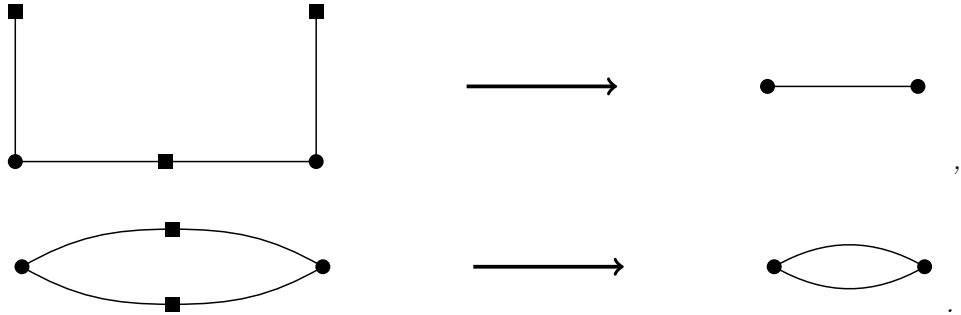
\begin{center}
\begin{tikzpicture}[square/.style={regular polygon,regular polygon sides=4}]

\draw[line width=.2mm] (-2,1) -- ++(0,-2) -- ++(2,0) -- ++(2,0) -- ++(0,2);
\node at (-2,1)[square,fill,inner sep=2pt]{};
\node at (-2,-1)[circle,fill,inner sep=2pt]{};
\node at (0,-1)[square,fill,inner sep=2pt]{};
\node at (2,-1)[circle,fill,inner sep=2pt]{};
\node at (2,1)[square,fill,inner sep=2pt]{};

\draw[line width=.5mm, ->] (4,0) -- (6,0);

\draw[line width=.2mm] (8,0) -- (10,0);
\node at (8,0)[circle, fill, inner sep = 2pt]{};
\node at (10,0)[circle, fill, inner sep = 2pt]{};

\end{tikzpicture}\,\,\,\, ,

\end{center}

\hskip .2cm

\begin{center}
\begin{tikzpicture}[square/.style={regular polygon,regular polygon sides=4}]

\draw[line width=.2mm] (-2,1) to[out=30,in=180] (0,1.5) to[out=0, in=150] (2,1) to[out=210, in=0] (0,0.5) to[out=180, in=330] (-2,1);
\node at (-2,1)[circle,fill,inner sep=2pt]{};
\node at (0,1.5)[square,fill,inner sep=2pt]{};
\node at (0,0.5)[square,fill,inner sep=2pt]{};
\node at (2,1)[circle,fill,inner sep=2pt]{};

\draw[line width=.5mm, ->] (4,1) -- (6,1);

\draw[line width=.2mm] (8,1) to[out=30,in=150] (10,1) to[out=210,in=330] (8,1);
\node at (8,1)[circle, fill, inner sep=2pt]{};
\node at (10,1)[circle, fill, inner sep=2pt]{};

\end{tikzpicture}\, .

\end{center}
\hskip .2cm

\item The image of the graph in Part $(3)$ is a triangle: 

\hskip .2cm

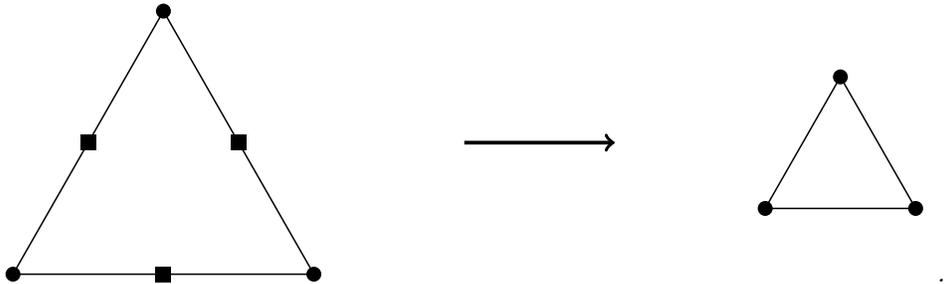
\begin{center}
\begin{tikzpicture}[square/.style={regular polygon,regular polygon sides=4}]

\draw[line width=.2mm] (-2,1) -- ++(2,0) -- ++(2,0) -- ++(-2,3.5) -- ++(-2,-3.5);
\node at (-2,1)[circle,fill,inner sep=2pt]{};
\node at (0,1)[square,fill,inner sep=2pt]{};
\node at (2,1)[circle,fill,inner sep=2pt]{};
\node at (0,4.5)[circle,fill,inner sep=2pt]{};
\node at (1,2.75)[square,fill,inner sep=2pt]{};
\node at (-1,2.75)[square,fill,inner sep=2pt]{};

\draw[line width=.5mm, ->] (4,2.75) -- (6,2.75);

\draw[line width=.2mm] (8,1.875) -- (10,1.875) -- (9,3.625) -- (8,1.875);

\node at (8,1.875)[circle, fill, inner sep = 2pt]{};
\node at (10,1.875)[circle, fill, inner sep = 2pt]{};
\node at (9,3.625)[circle, fill, inner sep = 2pt]{};

\end{tikzpicture}\, .

\end{center}
\hskip .2cm

\item The image of the graph in Part $(4)$ is a quadrangle:

\hskip .2cm
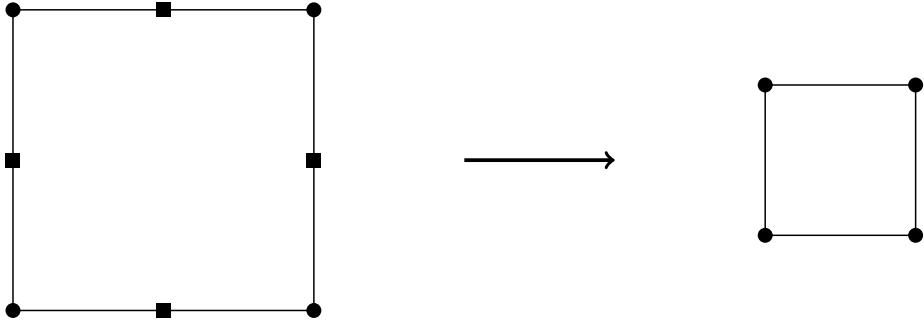
\begin{center}
\begin{tikzpicture}[square/.style={regular polygon,regular polygon sides=4}]

\draw[line width=.2mm] (-2,1) -- ++(0,-2) -- ++(2,0) -- ++(2,0) -- ++(0,2) -- ++(0,2) -- ++(-2,0) -- ++(-2,0) -- ++(0,-2);
\node at (-2,1)[square,fill,inner sep=2pt]{};
\node at (-2,-1)[circle,fill,inner sep=2pt]{};
\node at (0,-1)[square,fill,inner sep=2pt]{};
\node at (2,-1)[circle,fill,inner sep=2pt]{};
\node at (2,1)[square,fill,inner sep=2pt]{};
\node at (-2,3)[circle,fill,inner sep=2pt]{};
\node at (0,3)[square,fill,inner sep=2pt]{};
\node at (2,3)[circle,fill,inner sep=2pt]{};

\draw[line width=.5mm, ->] (4,1) -- (6,1);

\draw[line width=.2mm] (8,0) -- (10,0) -- (10,2) -- (8,2) -- (8,0);

\node at (8,0)[circle, fill, inner sep = 2pt]{};
\node at (10,0)[circle, fill, inner sep = 2pt]{};
\node at (10,2)[circle, fill, inner sep = 2pt]{};
\node at (8,2)[circle, fill, inner sep = 2pt]{};

\end{tikzpicture}\, .

\end{center}
\end{enumerate}
\end{example}

\begin{lemma}\label{lem:GammaisaGKMgraph} Let $(\bar{\Gamma}_M, \bar{\alpha}_M)$ be the odd GKM$_3$ graph corresponding to the closed, non-negatively curved GKM$_3$ manifold, $M^{2n+1}$.
 Let $(\Gamma, \alpha)$ be the graph obtained from  $(\bar{\Gamma}_M, \bar{\alpha}_M)$. Then $(\Gamma,  \alpha)$ admits a connection, $\nabla$, such that  $(\Gamma, \alpha, \nabla)$ is an abstract GKM$_3$ graph.
\end{lemma}

\begin{proof} 

Our goal is to show that $\Gamma$  is an abstract GKM$_3$ graph. Recall that by Definition \ref{defn:GKMkgraph}, a GKM$_3$ graph consists of a triple $(\Gamma, \alpha, \nabla)$  such that $\Gamma$ satisfies Properties $1$ and $2$ of Definition \ref{defn:GKMkgraph}, the axial function $\alpha$ satisfies Property $4$ of Definition \ref{defn:GKMkgraph}, and that a connection, $\nabla$, exists and satisfies Property $3$ of Definition \ref{defn:GKMkgraph}.

We begin by showing that $\Gamma$ satisfies Properties $1$ and $2$ of Definition \ref{defn:GKMkgraph}.
To prove Property $1$, 
we first claim that the graph $\Gamma$ obtained from $\bar{\Gamma}_M$ is $m$-valent, for some $m\leq n$. By Proposition \ref{props}, the odd-dimensional graph $\bar{\Gamma}_M$ is connected. So, to prove this claim
 it suffices to show that given two circles, $c_1, c_2\in V_{{\scaleto{{\scaleto{\circ}{4pt}}}{4pt}}}(\bar{\Gamma}_M)$,  joined by a square, $s_0$, the number of floating edges emanating from $c_1$ and $c_2$ is the same. 
Let $e, f\in E_{c_1}(\bar{\Gamma}_M)$, such that $e$ contains the square $s_0$, and $f$ is a floating edge with square $s$. Observe that because of the GKM$_3$ condition, $e$ and $f$ uniquely determine a $2$-dimensional face of $\bar{\Gamma}_M$.  Moreover, due to the linear dependence condition on the connection, it follows that $(\bar{\nabla}_M)_{c_1, c_2, s_0}(s)$ is a square belonging to the same $2$-dimensional face as $s$. 
  Now, Theorem \ref{graphs} tells us that for odd GKM$_3$ graphs corresponding to odd GKM$_3$ manifolds of non-negative curvature, there is only one  graph of a $2$-dimensional face that has a floating edge. So,  this graph has two floating edges, two circle vertices, and one grounded edge. This then implies that the connection sends a $1$-valent square to a $1$-valent square. Since the graphs of all the other $2$-dimensional faces have no floating edges, it follows that the connection sends   a $2$-valent square to a $2$-valent square.   We may conclude that at every circle in $\bar{\Gamma}_M$ there are exactly the same number of floating edges, and thus all vertices in the graph $\Gamma$ have the same valency, and Property $1$  holds.
 Note that the weights of the graph $\Gamma$ are still $3$-independent, and so Property $2$ holds.  

By the definition of the axial function $\alpha(e):= \bar{\alpha}_M(s)$, where $\pi(s)=e$, we see that Property $4$ of Definition \ref{defn:GKMkgraph} also holds.

We now need to show that we have a connection on $\Gamma$ that satisfies Property $3$.
In order to do so, we observe that we still have a notion of $2$-dimensional faces in $\Gamma$. Namely, given two edges $e,f$ attached to some vertex $v$ in $\Gamma$, there is a unique two-dimensional face in the odd-dimensional graph, $\bar{\Gamma}_M$, containing the edges corresponding to $e$ and $f$.  Moreover, this two-dimensional face in  $\bar{\Gamma}_M$ has no floating edges, since the graphs with floating edges in Theorem \ref{graphs} do not survive to graphs with a two-dimensional face, as can be seen in Example   \ref{examples}.
 We claim that this gives a well-defined connection on $\Gamma$, 
by sliding edges along edges inside these two-dimensional faces in the usual fashion. In fact,  we may directly translate the conditions in Definition \ref{GKMgraph} satisfied by the connection on $\bar{\Gamma}_M$, by making the following substitutions:
 $$(\bar{\nabla}_M)_{c_1, c_2, s_0}(s)\mapsto \nabla_e(f), $$
  where 
 $e=\pi(s_0)$, and $f=\pi(s)$.  
 With this definition, it is straightforward to verify that the connection satisfies both Definition \ref{defnconnection} and Property $3$ of Definition  \ref{defn:GKMkgraph}. 
\end{proof}

\begin{remark}  Lemma \ref{lem:GammaisaGKMgraph} tells us that  the number of floating edges in $\bar{\Gamma}_M$ is independent of the vertex. This fact is reflected in the statement of Theorem \ref{main1}. 
\end{remark}

We now show how to express the equivariant cohomology of the geometric odd GKM$_3$ graph $(\bar{\Gamma}_M, \bar{\alpha}_M)$ in terms of the equivariant cohomology of the abstract graph $(\Gamma, \alpha)$ obtained from $(\bar{\Gamma}_M, \bar{\alpha}_M)$. Recall that for an $n$-valent,  abstract GKM graph $(\Gamma,\alpha,\nabla)$, we say it is {\em orientable} if $H^{2n}(\Gamma,\alpha) \neq 0$.

In particular, for odd GKM$_3$ graphs, we  prove the following result. 

\begin{theorem}\label{main1} Let $M^{2n+1}$ be a closed, non-negatively curved odd GKM$_3$ manifold and $(\bar{\Gamma}_M, \bar{\alpha}_M)$ its graph. Let $k$ be the number of floating edges at a vertex in $\bar{\Gamma}_M$, let $(\Gamma, \alpha)$ be the  graph  obtained from $(\bar{\Gamma}_M, \bar{\alpha}_M)$,   and assume that $(\Gamma, \alpha)$ is orientable.
Then we have
\[
H^*_T(M) \cong H^*_T(\Gamma,\alpha) \otimes H^*(S^{2k+1})
\]
as $S({\mathfrak{t}}_{\mathbb{Q}}^*)$-algebras, where the $S({\mathfrak{t}}_{\mathbb{Q}}^*)$-algebra structure on $H^*_T(\Gamma,\alpha) \otimes H^*(S^{2k+1})$ is the tensor product of the standard $S({\mathfrak{t}}_{\mathbb{Q}}^*)$-algebra structure on $H^*_T(\Gamma,\alpha)$ and the trivial one on  $H^*(S^{2k+1})$.  
Therefore, we obtain
\[
H^*(M) \cong H^*(\Gamma,\alpha)\otimes H^*(S^{2k+1}).
\]
\end{theorem}

\begin{proof}[Proof of Theorem \ref{main1}]
By Proposition \ref{valence}, any square in the odd GKM graph of $M$ has valence one or two. We denote by $V_{\scaleto{\square}{4pt}}^1$ and $V_{\scaleto{\square}{4pt}}^2$ the sets of squares with valence one and two, respectively. For a square $s\in V_{\scaleto{\square}{4pt}}^1$ we denote the unique circle connected to $s$ by $c(s)$, whereas for $s\in V_{\scaleto{\square}{4pt}}^2$ we denote the two circles by $c_1(s)$ and $c_2(s)$, with any ordering.

Then  Displays \eqref{D2} and \eqref{D3} reduce to the following divisibility relations for $P_c$ and $Q_c$:
\begin{equation}\label{Div1}
{ P_{c_1(s)} \equiv P_{c_2(s)} \text{ mod } \alpha(s), \text{ and }  Q_{c_1(s)} \equiv \pm Q_{c_2(s)}  \text{ mod } \alpha(s) \text{ for }s\in V_{\scaleto{\square}{4pt}}^2}, \text{ and }
\end{equation}
\begin{equation}\label{Div2}
Q_{c(s)} \equiv 0 \text{ mod } \alpha(s) \textrm{ for } s\in V_{\scaleto{\square}{4pt}}^1.
\end{equation}
Then the equivariant cohomology of $M$ is given by
$$
H^*_T(M) \cong \{(P_c + Q_c\theta)_{c\in V_{\scaleto{{\scaleto{\circ}{4pt}}}{4pt}}} \left|  
P_c, Q_c \textrm{ satisfy Relations \eqref{Div1} and  \eqref{Div2}} \}.\right.
$$

By comparing the divisibility relations in Displays \eqref{Div1} and \eqref{Div2} for $H^*_T(M)$  and the description of the equivariant cohomology in terms of the GKM graph in even dimensions in Display \eqref{D2},  we see that the divisibility relations imposed by grounded edges in the odd-dimensional GKM graph $\bar{\Gamma}_M$ are precisely those imposed by the edges in $\Gamma$. 
As a consequence, we see that we obtain an isomorphism 
\begin{equation}\label{evenec}
\psi:H^{even}_T(M) \longrightarrow H^*_T(\Gamma,\alpha, \nabla).
\end{equation} Moreover, we also see by Proposition \ref{Ecohom} and Definition \ref{defn:cohomgraph}  that   $H^{even}(M) = H^*(\Gamma,\alpha, \nabla)$.

For cohomology in odd dimensions, note that the $Q_c$ have to be divisible by all $k$ weights of the floating edges at $c$, so 
 $\mathrm{deg}(Q_c)\geq k$.  Thus $H^{2l+1}_T(\Gamma, \alpha) = 0$ for $l<k$.
On the other hand, as $\Gamma$ is an orientable $(n-k)$-valent graph,   we have that $H^{2(n-k)}(M) = H^{2(n-k)}(\Gamma, \alpha)\neq 0$.  Poincar\'e duality now implies that $H^{2k+1}(M) \neq 0$.  By Proposition \ref{Ecohom}, $H^*_T(M)\rightarrow H^*(M)$ is onto, and  it follows that $H^{2k+1}_T(M)\neq 0$. 
So there  exists a nonzero element $\omega\in H^{2k+1}_T(M)$. 

Our goal is now to show that we may multiply elements of the even dimensional cohomology by $\omega$ to obtain an $S({\mathfrak{t}}_{\mathbb{Q}}^*)$-module isomorphism from $H^{even}_T(M)$ to  $H^{odd}_T(M)$.  Using the divisibility relations in Display \eqref{Div2}, we may express any nontrivial element $\omega$ in $H^{2k+1}_T(M)$ as
$$\omega  =(\omega_c\theta)_{c\in V_{\scaleto{{\scaleto{\circ}{4pt}}}{4pt}}} \in H^{2k+1}_T(M),$$
where
\begin{equation}\label{omegac}
\omega_c   = a_c \, \alpha(s_1(c))\cdot \alpha(s_2(c))\cdot \hdots \cdot  \alpha(s_k(c)),
\end{equation}
$a_c\in \mathbb{Q}$ depends only on the circle $c$, and  $s_1(c),\hdots, s_k(c)$  are the  squares in the $k$ floating edges connected to $c$.   
The $\omega_c$ must also satisfy the first set of divisibility relations  in Display \eqref{Div1}, and since $\omega\in H^{2k+1}_T(M)$,  this is equivalent to requiring $\omega_{c_1(s)}\equiv\pm \omega_{c_2(s)} \mod \alpha(s)$, for $s\in V_{\scaleto{\square}{4pt}}^2$.

We now want to show that $a_c\neq 0$ for all $c \in V_{{\scaleto{{\scaleto{\circ}{4pt}}}{4pt}}}$.   
We will argue by contradiction. Assume  that  $a_{c'}=0$ for some $c'$. Since $\omega$ is non-trivial, there is a $c''\neq c'$ such that $a_{c''}\neq 0$. 
 The connectivity of the graph implies that there must be some $s\in V_{\scaleto{\square}{4pt}}^2$, such that $a_{c_1(s)}=0$, but  $a_{c_2(s)}\neq 0$. Then $0 = \omega_{c_1(s)}$ by Equation \eqref{omegac}. But 
then the divisibility relation in \eqref{Div1} gives us that $\omega_{c_1(s)} \equiv \omega_{c_2(s)}\equiv 0\, \text{ mod }\alpha(s)$, and 
Equation \eqref{omegac}  tells us that $\alpha(s)$ is then a scalar multiple of one of the  $\alpha(s_i)$, $1\leq i\leq k$.  However Part $2$ of  the definition of a GKM-manifold tells us 
$\alpha(s)$ and  $\alpha(s_i)$ are pairwise linearly independent for $1\leq i\leq k$, a contradiction.

We now claim that $a_c\neq 0$ for all $c \in V_{{\scaleto{{\scaleto{\circ}{4pt}}}{4pt}}}$ implies that $\dim(H^{2k+1}_T(M)) = 1$.  Suppose instead that we have two linearly independent elements $\mu, \omega\in H^{2k+1}_T(M)$, where $\mu_c$ and $\omega_c$ are as in Equation \eqref{omegac}. Then for some $c$,  let $\eta_c=\gamma\omega_{c}-\mu_c$, with $\gamma=a_c^{\mu}/a_c^{\omega}$. But then for this same $c$, $\eta_c=0$, contradicting the fact that for any  element of $H^{2k+1}_T(M)$, $a_c\neq 0$ for all $c \in V_{{\scaleto{\circ}{4pt}}}$.

 It now follows that multiplication with $\omega$ defines an  $S({\mathfrak{t}}_{\mathbb{Q}}^*)$-module injection $H^{even}_T(M)\to H^{odd}_T(M)$, that is,  
 
 \begin{equation}\label{odd}
 \begin{aligned} \phi: \, H^{even}_T(M) &\longrightarrow H^{odd}_T(M)\\
 \psi &\longmapsto  \psi \cdot \omega,
 \end{aligned}
 \end{equation}
  with $\omega \in H^{2k+1}_T(M) $. We claim that $\phi$ is an isomorphism. To do so, we must show that $\phi$ is onto. 
  Let $Q=(Q_c\theta)_{c\in V_{\scaleto{\circ}{4pt}}}\in H^{odd}_T(M)$. 
For each $c\in V_{\scaleto{\circ}{4pt}}$, the polynomial $Q_c$ is divisible by any $\alpha(s)$ with $c(s) = c$ for $s \in V_{\scaleto{\square}{4pt}}^1$, and so we can write 
$Q=(Q_c \theta) =  (P_c \,\omega_c \theta)=P\cdot \omega$,  for some $P=(P_c)$.
  Then to show that $\phi$ is onto,  we have to show that $(P_c)_{c \in V_{\scaleto{\circ}{4pt}}}\in H^{even}_T(M)$. That is, we must verify that the $P_c$ satisfy the divisibility relation \eqref{Div1} for all $c\in V_{\scaleto{\circ}{4pt}}$. Let $s'\in V_{\scaleto{\square}{4pt}}^2$ be arbitrary. Then $\alpha(s')$ divides both $Q_{c_1(s')}\pm Q_{c_2(s')}$ and  $\omega_{c_1(s')} \pm \omega_{c_2(s')}$, where $\pm$ is taken to be the same sign in both expressions. Then we compute
\begin{align*}
Q_{c_1(s)}  \pm Q_{c_2(s)}  &= P_{c_1(s)} \omega_{c_1(s)} \pm P_{c_2(s)} \omega_{c_2(s)} \\
&= (P_{c_1(s)} - P_{c_2(s)})\, \omega_{c_1(s)} + P_{c_2(s)}(\omega_{c_1(s)} \pm \omega_{c_2(s)}),
\end{align*}
and the divisibility assumptions imply that $\alpha(s')$ divides $(P_{c_1(s')} - P_{c_2(s')})\, \omega_{c_1(s')}$. By the same argument used to show that $a_c\neq 0$ for all $c\in V_{\scaleto{\circ}{4pt}}$, $\alpha(s')$ does not divide $\omega_{c_1(s')}$, since $s'\in V_{\scaleto{\square}{4pt}}^2$. So $\alpha(s')$ has to divide $P_{c_{1}(s')}-P_{c_{2}(s')}$, which is precisely the divisibility relation \eqref{Div1}. So, we have shown that $(P_c)_{c\in V_{\scaleto{\circ}{4pt}}}\in H^{even}_T(M)$.  Hence, $\phi$ is onto and multiplication by $\omega$ defines an $S({\mathfrak{t}}_{\mathbb{Q}}^*)$-module isomorphism.

 We can now prove that $H^*_T(M)\cong H^*_T(\Gamma,\alpha)\otimes H^*(S^{2k+1})$, by extending  
 the isomorphism in Display \eqref{evenec} via the isomorphism in Display \eqref{odd} 
 to an $S({\mathfrak{t}}_{\mathbb{Q}}^*)$-algebra isomorphism
\begin{align*}
\Psi : \, & H^*_T(\Gamma,\alpha)\otimes H^*(S^{2k+1}) \longrightarrow H^*_T(M)\,\cr
 &\gamma\otimes id + \beta\otimes\, \mu_{S^{2k+1}} \longmapsto  \psi(\gamma) + \phi(\psi(\beta))= \psi(\gamma) + \psi(\beta)\omega,
\end{align*}
where $\gamma, \beta\in H^*_T(\Gamma,\alpha)$, and $\mu_{S^{2k+1}}$ is the volume form of $S^{2k+1}$.

The second statement of the theorem then follows immediately from Proposition \ref{Ecohom} and Definition \ref{defn:cohomgraph} because the $S({\mathfrak{t}}_{\mathbb{Q}}^*)$-module structure on $H^*(\Gamma,\alpha)\otimes H^*(S^{2k+1})$ only exists on the first factor.
\end{proof}

Theorem \ref{main} is now a direct consequence of Theorem \ref{main1}.

\subsection{Orientability of the Associated Graph ${\bf \Gamma}$}\label{ss5}

Without the assumption of orientability of the associated graph, the conclusion of the Main Theorem \ref{main1} does not hold. 
We include an example here.
Consider
$$M^9= S^2(1)\times S^2(1)\times S^2(1)\times S^3(1)$$
with the product metric and with a $T^4$-action given as a product of four $T^1$-actions, where the action on each $S^2$ is by rotation, and  on $S^3\subset \mathbb{C}^2$ is fixed point homogeneous, namely, 
$$(e^{i\theta}, (z_1, z_2))\mapsto (e^{i\theta}z_1, z_2).$$
Then consider the $\mathbb{Z}_2$-action on $M$, given by the antipodal map on each $S^2$ and on $S^3$ by $$(z_1, z_2)\mapsto (z_1, \bar{z}_2).$$

As the $\mathbb{Z}_2$-action on $M$ commutes with the $T^4$ action and is free, orientation-preserving, and by isometries,  
 the quotient
$$N^{4k+1} =M/(\mathbb{Z}_2)$$
is a non-negatively curved, closed, orientable manifold with the induced $T^4$-action.

Using the transfer isomorphism (see, for example, Theorem III.2.4 in \cite{Br}), we obtain $H^*(N) = H^*(M)^{\mathbb{Z}_2}$,  and we compute the Betti numbers of $N$ to be
\[
b_i(N) = \begin{cases}1 &\, i = 0,9\\ 3 & \, i = 4,5 \\ 0 &{\textrm{ otherwise}}.\end{cases}
\]

Both $M$ and $N$ are orientable, however, 
the associated GKM graph of $N$ is the quotient of a $4$-dimensional cube, $\text{I}^4/\mathbb{Z}_2$, which is not orientable.

We claim that the $T^4$-action on $N$ is odd GKM$_4$, and hence odd GKM$_3$.
First, the fixed point set of the $T^4$-action on $N$ consists of exactly $4$ circles.
 Second, the condition on the weights is satisfied, since  the action of $T^4$ on $M$ is  odd GKM$_4$, and the condition on the weights of the isotropy representations at the fixed points of $N$ is inherited from $M$. 
 Third,  we can use Proposition \ref{eformal} to see that the $T^4$-action is equivariantly formal: as computed above, the sum of the Betti numbers of $N$ is equal to $8$, as is the sum of the Betti numbers of $N^{T^4}$. 
 So, the claim holds.  
 However, the cohomology ring of $N$ does not split off the cohomology of an odd-dimensional sphere, and so  if we remove the hypothesis on the orientability of  the associated graph $\Gamma$, we see that the conclusion of Theorem \ref{main} does not hold for $N$.

Note that in this example $N$ is not simply-connected, and so we pose the following question:

\begin{question} If $N$ is a closed, simply-connected, odd GKM$_3$ manifold, is the associated graph $\Gamma$ orientable? 
\end{question}

\subsection{Applications of Theorem \ref{main1}} \label{5.2}
We consider some special subcases of Theorem \ref{main1}. Firstly, if the metric on the GKM$_3$ manifold is positively curved, then by Theorem \ref{Rong}, every two-dimensional face of the GKM graph has only one circle. This implies that the GKM graph, $\bar{\Gamma}_M$, is the pinwheel depicted in Example \ref{ex:pinwheel}, and so $\Gamma$ is a single vertex. In particular, a single vertex graph is orientable. Theorem \ref{positive} is then immediate.  
 
 As indicated in the Introduction, the proof of Theorem \ref{2} follows from the proof of Theorem \ref{GKM3+}.
 
Theorem \ref{3} follows in the same way using Theorem \ref{GKM40}. In order to apply Theorem \ref{GKM40}, we need to verify that $\Gamma$, the GKM$_4$ graph obtained from the odd-dimensional GKM graph of $M$, $\bar{\Gamma}_M$, is a graph with {\em small three-dimensional faces} (see Definition 3.5 in \cite{GW2}). As noted in \cite{GW2}, a GKM$_4$ graph that has two-dimensional faces with at most $4$ vertices must have small three-dimensional faces. Since $M$ has non-negative curvature, Theorem \ref{graphs} tells us that the two-dimensional faces of $\bar{\Gamma}_M$ have at most $4$ circle vertices, and hence the two-dimensional faces of $\Gamma$ have at most $4$ vertices. The result follows.

\section{Invariant almost contact structures}\label{s6}

The goal of this section is to prove Theorem \ref{AC} of the Introduction.  We begin by recalling the definition of an almost contact structure.
\begin{definition}[{\bf Almost contact structure}]
An {\em almost contact structure}  $(\phi,\xi,\eta)$ on a $(2n+1)$-manifold $M$ consists of a $(1,1)$-tensor field $\phi$, a vector field $\xi$, and a differential one-form $\eta$ such  that 
 $$\eta(\xi)=1 \, \textrm{  and } \,
\phi^2(X) = -X + \eta(X) \xi,$$
for any vector field $X$ on $M$. Note that the vector field $\xi$, which is called {\em the Reeb vector field}, is uniquely determined by $\phi$ and $\eta$, namely at a point $p$ it is the unique vector $\xi_p$ such that $\phi_p(\xi_p)=0$ and $\eta_p(\xi_p)=1$.
\end{definition}

Given an almost contact structure $(\phi,\xi,\eta)$ on $M^{2n+1}$, if $N^{2k+1}\subset M^{2n+1}$ is a submanifold such that $\xi$ is tangent to $N$, and $\phi$ restricts to a well-defined tensor field on $N$, then $(\phi,\xi,\eta)$ restricts to an almost contact structure on $N$. In this case we call $N$ an {\em almost contact} submanifold.

The following lemma may be well-known, but could not be located by the authors elsewhere in the literature. For
completeness, a proof is presented here.

\begin{lemma}\label{lem:fixedalmostcontact} Let  $(\phi,\xi,\eta)$ be an almost contact structure on a manifold $M$, invariant under the action of a compact Lie group $G$. Then every connected component of the fixed point set $M^G$ of the action is an almost contact submanifold.
\end{lemma}
\begin{proof}
Let $N$ be a connected component of $M^G$ and recall $N$ 
is an embedded submanifold of $M$ by work of Kobayashi \cite{Kobayashi}.  Then for every point $p\in N$
, the tangent space of $N$ is given by
\[
T_pN =(T_pM)^G,
\]
the set of vectors fixed by the isotropy representation of $G$ at $p$.  The $G$-invariance of $\phi$ and the fact that $p$ is fixed by $G$ then implies that $\phi$ maps $T_pN$ to itself. For the same reasons, it follows that $\xi$ is tangent to $N$. Thus, the connected components of $M^G$ are almost contact submanifolds.
\end{proof}
Combining Lemma \ref{lem:fixedalmostcontact} and the fact that the almost contact structure gives us a $T$-invariant almost complex structure on $\ker \eta_p$, we obtain the following proposition.
\begin{proposition}\label{6.3} Let $M^{2n+1}$ be an odd GKM manifold with a $T$-invariant almost contact structure $(\phi,\xi,\eta)$. 
Then the following are true.
\begin{enumerate}
\item Every component of the fixed point set of $T$ is an isolated, closed flow line of $\xi$.
\item At any fixed point $p$ of the torus action, the weights of the isotropy representation at $p$ are well-defined elements of ${\mathfrak{t}}_{\mathbb{Q}}^*$.
\end{enumerate}
\end{proposition}
\begin{proof}
By the definition of an odd GKM action, every component of the fixed point set of $T$ is an isolated circle.
 Lemma \ref{lem:fixedalmostcontact} then gives us that the restriction of $(\phi,\xi,\eta)$ to any of these isolated circles is an almost contact submanifold, and the first statement follows. 
 
To prove the second statement, we note that  $\phi$ defines a $T$-invariant almost complex structure on $\ker \eta_p$. Since the almost contact structure is $T$-invariant, we have $\eta_{tp}(dt_p(v)) = 
\eta_p(v)$ for all $v \in T_pM$ and $t \in T$. So at a fixed point 
$p$, $\ker (\eta_p)$ is $T$-invariant. This fact combined with the fact that for each point the tangent space to $\ker (\eta_p)$ has a complex structure, gives us a {\em complex} $T$-representation at $p$, and the second statement follows.
\end{proof}

By Proposition \ref{6.3},  in the presence of a  $T$-invariant almost contact structure on $M$, we have that the weights are well-defined elements of ${\mathfrak{t}}_{\mathbb{Q}}^*$.  This then allows us to slightly modify the odd GKM graph, $\bar{\Gamma}_M$, of the $T$-action on $M$, which we call a \emph{signed odd GKM graph}, 
as follows. We consider the same underlying graph, $\bar{\Gamma}_M$, but now we assign weights to edges, not to squares, that is, to each edge connecting a circle $c$ to a square $s$, we assign the corresponding weight of the isotropy representation at the circle $c$, which is an element in ${\mathfrak{t}}_{\mathbb{Q}}^*$. Regarded modulo $\pm 1$, this weight is the same as the weight assigned to the square $s$ in the original odd-dimensional graph, $\bar{\Gamma}_M$.  If, in addition, the signed weights on the edges emanating from a square sum to $0$, we call such a graph {\em alternating}. This leads us to make the following definition.
\begin{definition}\label{alt/bal} If the signed odd GKM graph induced from the invariant almost contact structure on the odd GKM manifold is alternating, then we  say that the almost contact structure is {\em alternating}.
\end{definition}

The connection of a signed odd GKM graph is modified as follows (cf. Definition \ref{GKMgraph}).
Formally, if we denote the set of edges emanating from a circle $c$ by $E(c)$, then the axial function $\bar{\alpha}_M$ is a collection of maps $E(c)\to {\mathfrak{t}}_{\mathbb{Q}}^*$, for all $c$. The connection can be regarded as a collection of maps ${(\bar{\nabla}}_M)_{c_1,c_2,s_0}:E(c_1)\to E(c_2)$, where $c_1,c_2\in V_{\scaleto{\circ}{4pt}}(s_0)$, and it satisfies that for every edge $e\in E(c_1)$ there exists a constant $c\in {\mathbb{Z}}$ such that
\[
\bar{\alpha}_M((\bar{\nabla}_M)_{c_1,c_2,s_0}(e)) = \bar{\alpha}_M(e) + c\bar{\alpha}_M(e_0),
\]
where $e_0$ is an edge connecting $c_1$ or $c_2$ with $s_0$. 
\begin{remark} If $M^{2n}$ is a closed, non-negatively curved GKM$_k$ manifold admitting an invariant almost complex structure, then the associated classical GKM graph is a {\em signed} GKM graph (see Remark \ref{remark2.18}). At the beginning of Subsection \ref{5.1} we showed how one may obtain an abstract GKM graph from an odd GKM graph whose squares have valence less than or equal to two. 
It follows immediately that we can obtain classical {\em signed} GKM graphs from {\em alternating} odd GKM graphs. 
\end{remark}

We now restate a result from the proof of Lemma 5.6 in \cite{GW1}, noting that this result is independent of curvature.
\begin{lemma}\cite{GW1}\label{5.6} Let $\Gamma$ be an abstract, signed GKM$_3$ graph. Then $\Gamma$ admits no biangles.
\end{lemma}
Recall that we denote by $\Sigma^k$ the
orbit space of the linear, effective action of the $k$-dimensional torus on $S^{2k}$.
The following corollary to Lemma \ref{5.6} is immediate (cf. the proof of Lemma 7.1 in \cite{GW2}).
\begin{corollary}\label{cor5.6} Let $\Gamma$ be an abstract,  signed GKM$_3$ graph. Then there are no maximal simplices in $\Gamma$ 
 with the combinatorial type of $\Sigma^k$.
\end{corollary}

Before we prove Theorem \ref{AC},  we first recall the definition of a Bott manifold.
\begin{definition}[{\bf Generalized Bott manifold}] We say that a manifold $X$ is a {\em generalized Bott manifold} if it is the total space of an iterated $\ccc P^{n_i}$-bundle
$$X =X_k\rightarrow X_{k-1}\rightarrow\cdots\rightarrow X_1\rightarrow X_0=\{pt\},$$
where each $X_i$ is the total space of the projectivization of a Whitney sum of $n_i + 1$ complex line bundles over $X_{i-1}$. 
\end{definition}
\begin{remark} Torus manifolds over $\prod \Delta^{n_i}$, where $\Delta^{n_i}$ denotes the standard simplex of dimension $n_i$, admitting an invariant almost complex
structure were classified in Choi, Masuda, and Suh \cite{CMS}. They are all diffeomorphic to the so-called generalized Bott manifolds. 
\end{remark}
We are now ready to prove Theorem \ref{AC}. 
\begin{proof}[Proof of Theorem \ref{AC}]

Let $M^{2n+1}$ be a closed, non-negatively curved, odd GKM$_4$ manifold which admits an invariant almost contact structure that is alternating. By assumption, $\bar{\Gamma}_M$ is alternating and so the abstract GKM graph obtained from it is signed.
By Corollary \ref{cor5.6}, a signed, abstract GKM$_3$ graph has no maximal simplices 
with the combinatorial type of $\Sigma^k$.
Since GKM$_4$ manifolds are also GKM$_3$, it follows that $\Gamma$ contains no maximal simplices with the combinatorial type of $\Sigma^k$.

We can now argue as in Section 7 of \cite{GW2}  to obtain the result.
We briefly outline the proof here for the sake of completeness.  As noted at the end of Section \ref{5.2}, non-negative curvature and the GKM$_4$ condition guarantee that $\Gamma$ will have small three-dimensional faces.  So we may apply Theorem 3.11 of \cite{GW2} to show that $\Gamma$ is finitely covered by a graph, $\tilde{\Gamma}$, which is   the vertex-edge graph of a finite product of
simplices.   One then  shows that the quasitoric manifold corresponding to the graph $\tilde{\Gamma}$ admits an invariant complex structure  in Theorem 7.1 of \cite{GW2}.  Applying Theorem 6.4 of \cite{CMS}, then shows us that $\tilde{\Gamma}$ is the GKM graph of a generalized Bott manifold. Finally, we use Theorem 
7.5 of \cite{GW2}, 
to show that $\tilde{\Gamma}=\Gamma$. Thus, we may apply the GKM theorem to conclude that
the rational cohomology ring of $M$ is the tensor product of the rational cohomology ring of an odd-dimensional sphere and the rational cohomology ring of a generalized Bott manifold, 
as desired.
\end{proof}

It seems very likely that the graphs corresponding to non-negatively curved, odd GKM$_3$ manifolds admitting an invariant almost contact structure are alternating. We finish with the following conjecture.
\begin{conjecture} Let $M^{2n+1}$ be a closed,  non-negatively curved odd-dimensional GKM$_3$ manifold admitting an invariant almost contact structure. Then the odd GKM$_3$ graph corresponding to $M^{2n+1}$ is alternating.
\end{conjecture}


\end{document}